\documentclass[a4paper,12pt,showkeys]{amsart}

\usepackage{amsmath,amssymb,amsthm,mathrsfs}
\usepackage{type1cm}
\usepackage{ulem}
\usepackage{fancybox}
\usepackage[pdftex]{graphicx}
\usepackage{tikz}
\usepackage[all]{xy}
\usepackage{array,booktabs}
\usepackage{bm}

\theoremstyle{plain}
\newtheorem{theorem}{Theorem}[section]

\newtheorem{lemma}[theorem]{Lemma}

\theoremstyle{definition}
\newtheorem{definition}[theorem]{Definition}
\newtheorem{remark}[theorem]{Remark}
\newtheorem{example}[theorem]{Example}
\newtheorem{proposition}[theorem]{Proposition}

\newcommand{\GL}{{\rm{GL}}}
\newcommand{\PGL}{{\rm{PGL}}}

\newcommand{\Ker}{{\rm{Ker}}}
\newcommand{\im}{{\rm{Im}}}

\newcommand{\Hom}{{\rm{Hom}}}
\newcommand{\End}{{\rm{End}}}

\newcommand{\Lim}{\displaystyle \lim_{t\rightarrow 0}}
\newcommand{\PM}{\mathbb{P}M}

\begin{document}

\title{On $\mathbb{P}M$-monoids and braid $\mathbb{P}M$-monoids}
\author{Toshinori Miyatani}
\date{}
\address{Graduate School of Science, Hokkaido University, Sapporo, 004-0022, Japan}
\email{miyatani@math.sci.hokudai.ac.jp}

\begin{abstract}
In this paper, we shall introduce two monoids. One is called a $\mathbb{P}M$-monoid which contains the symmetric group, the other is called a braid $\PM$-monoid which contains the braid group. We shall develop the theory of $\PM$-monoids and that of braid $\PM$-monoids. The $\PM$-monoids is obtained in the context of the compactification of projective linear group defined by Mutsumi Saito. The structure of $\PM$-monoids is described in terms of matched pairs. We can define braid $\PM$-monoid using a presentation for the $\PM$-monoid. As main results, we show that braid $\PM$-monoids are described by geometric braids and we find a solution to the word problem for the braid $\PM$-monoids. 
\end{abstract}


\maketitle
\tableofcontents

\section{Introduction}
The braid groups and the symmetric groups have deep relations and have rich theories. The braid groups are generalized to the Artin groups and the symmetric groups to the Coxeter groups \cite{BK}. We will consider two monoids analogous to the symmetric groups and the braid groups, respectively. We first define a monoid $\mathscr{R}_n$, which we call a $\PM$-monoid, and the $\mathbb{P}M$-monoids contain the symmetric groups. The $\mathbb{P}M$-monoids can be seen as analogue to the look monoids defined by L. Solomon \cite{So}. The $\PM$-monoid is obtained in the context of the compactification of projective linear group defined by Mutsumi Saito \cite{S1}. 

The structure of the $\PM$-monoid is described by the matched pair of the symmetric group and the collection of the ordered partition (Proposition \ref{rm}). We show that the $\PM$-monoid has a presentation with generators and relations (Proposition \ref{rep}). This is an analogue of the fact that the look monoid has a presentation with generators and relations \cite{So}.
Using this presentation we define a braid $\PM$-monoid denoted by $\mathscr{RB}_n$ (Definition \ref{Bpm}). The braid $\mathbb{P}M$-monoid is an analogue to the inverse braid monoid defined by D. Easdown and T. G. Lavers \cite{EL}. As the main results, we show that the braid $\PM$-monoid has a presentation by geometric braids and contains the braid group (Theorem \ref{pmb}). This is an analogue of the fact that the braid groups and the inverse braid monoids have the presentation by the geometric braids \cite{KT},\cite{EL}. Moreover we shall find a solution to the word problem of the braid $\PM$-monoid (Theorem \ref{wp}). This statement is an analogue of the fact that the braid groups and the inverse braid monoids have a solution to the word problem \cite{GM},\cite{V1}.

This paper is organized as follows. In Section 2, we explain the compactification of the projective linear group defined by Mutsumi Saito. This section is also a survey of this compactification. In Section 3.1, we review the linear algebraic monoids. In Section 3.2, we define the $\PM$-monoid, and we reveal the structure. In Section 3.3, we study the properties of $\mathbb{P}M$-monoids. In Section 3.4, we construct a presentation for $\mathbb{P}M$-monoids. In Section 3.5, we review braid groups and  inverse braid monoids. In Section 3.6, we introduce the braid $\PM$-monoid. In Section 4.1, we will show that the braid $\PM$-monoids have a presentation by geometric braids. In Section 4.2, we find a solution to the word problem of the braid $\PM$-monoids. 

\section{Compactification of $\PGL$}
We explain the compactification of the projective linear group constructed by M. Saito \cite{S1}. 
\subsection{Motivation}
One strategy of compactification is constructing a ``limit''. Then we consider the set of all limit points and introduce a topology compatible with the limit. For instance Y. A. Neretin constructed a compactification of the projective linear group by this strategy called hinge \cite{N1}. 

Let $V$ be an $n$-dimensional vector space over $\mathbb{C}$ and $A_i\in \End(V)$, $(i=1,2,\dots)$. We define the linear map
\begin{equation} \label{Ae}
A_{\epsilon}:=\displaystyle \sum_{i=0}^{m} A_i \epsilon^i
\end{equation}
such that (\ref{Ae}) is in $\GL(V)$ for $\epsilon\neq 0$. Dividing by nonzero scalar matrices we consider the projective linear map
\begin{equation} \label{Ae-}
\overline{A_{\epsilon}}\in \PGL(V).
\end{equation}
We want to define a ``limit'' $\displaystyle \lim_{\epsilon\rightarrow 0}\overline{A_{\epsilon}}$. To define a limit, we observe the action of $\overline{A_{\epsilon}}$ on $\mathbb{P}(V)$. For $\overline{x}\in\mathbb{P}(V)$ we have 
\begin{equation*}
\displaystyle \lim_{\epsilon\rightarrow 0}\overline{A_{\epsilon}(x)}=
\begin{cases}
\overline{A_0x} & (x\not\in \Ker A_0) \\
\overline{A_1x} & (x\not\in \Ker A_0\backslash \Ker A_1) \\
\overline{A_2x} & (x\not\in \Ker A_0\cap \Ker A_1\backslash \Ker A_2) \\
\vdots  & \,\,\,\,\,\,\,\,\,\,\,\,\,\,\,\,\,\,\,\,\,\,\,\,\,\,\,\,\,\,\,\,\,\,\,\,\,\,\,\,\,\,\,\,\,\,\,\,\,\,\,\,\,\,\,\,\,\,\,\,\,\,\,\,\,\,\,\,\,\,\,\,\,\,\,.
\end{cases}
\end{equation*}
Thus we define the limit of (\ref{Ae-}) as
\begin{equation} \label{li}
\displaystyle \lim_{\epsilon\rightarrow 0}\overline{A_{\epsilon}}:=(\overline{A_0},\overline{A_1}|_{\mathbb{P}(\Ker A_0)},\overline{A_2}|_{\mathbb{P}(\Ker A_0\cap \Ker A_1)},\dots).
\end{equation}
\subsection{Definition of $\mathbb{P}M$}
In order to construct a compactification of the projective linear group, we consider the set of the form of the right hand side of (\ref{li}). We define the following sets.
Let $V$ be an $n$- dimensional vector space over $\mathbb{C}$. Set
\begin{equation*}
M:=M(V)=\left\{(A_0,A_1,\dots,A_m)\mid 
\begin{matrix} m=0,1,2,\dots \\ 0\neq A_i\in \Hom(V_i,V)\,\, (0\le i\le m) \\ V_0=V, V_{m+1}=0 \\ V_{i+1}=\Ker (A_i)\,\,(0\le i\le m)
\end{matrix}
\right\}
\end{equation*}
and
\begin{equation*}
\widetilde{M}:=\widetilde{M}(V)=\left\{(A_0,A_1,\dots,A_m)\mid 
\begin{matrix} m=0,1,2,\dots \\ 0\neq A_i\in \End(V)\,\, (0\le i\le m) \\ \displaystyle \cap_{k=0}^{i-1}\Ker A_k\not\subseteq \Ker A_i \\ \displaystyle \cap_{k=0}^{m}\Ker A_k=0
\end{matrix}
\right\}.
\end{equation*}
Let $\mathcal{A}:=(A_0,A_1,\dots,A_m)\in M$. Since $A_i\in \Hom(V_i,V)\backslash \{0\}$, we can consider the element $\overline{A_i}\in \mathbb{P}\Hom(V_i,V)$ represented by $A_i$, and we can define
\begin{equation*}
\mathbb{P}\mathcal{A}:=(\overline{A_0},\overline{A_1},\dots,\overline{A_m}).
\end{equation*}
Let $\mathbb{P}M=\mathbb{P}M(V)$ denote the image of $M$ under $\mathbb{P}$. $\mathbb{P}\widetilde{M}$ can be defined in the similar way.

\subsection{Topology of $\mathbb{P}M$}
We introduce a topology in $\PM$ which we can deal with the limit (\ref{li}). We fix a Hermitian inner product on $V$. Let $W$ be a subspace of $V$. By considering $V=W\oplus W^{\bot}$ via this inner product, we regard $\Hom(W,V)$ as a subspace of $\End(V)$. We consider the classical topology in $\mathbb{P}\Hom(W,V)$ for any subspace $W$ of $V$. \\
Let $\mathbb{A}=(A_0,A_1,\dots,A_m)\in M$. Then $A_i\in\Hom(V_i,V)$, where $V_i=V(\mathbb{A})_i=\Ker(A_{i-1})$. Let $U_i$ be a neighborhood of $\overline{A_i}$ in $\mathbb{P}\Hom(V(\mathbb{A})_i, V)$. Then set 
\begin{equation} \label{nbd}
U_{\mathbb{PA}}(U_0,\dots,U_m)
=\left\{\mathbb{PB}=(\overline{B_0},\overline{B_1},\dots,\overline{B_n})\mid 
\begin{matrix} ^{\forall}i=1,\dots,m, ^{\exists}j\in \{1,\dots,n\} \text{ s.t. } \\ V(\mathbb{B})_j\supseteq V(\mathbb{A})_i \text{ and } 
\\ \overline{B_j|_{V(\mathbb{A})_i}}\in U_i
\end{matrix}
\right\}.
\end{equation}
We will explain why the sets (\ref{nbd}) define the topology that can deal with the limit (\ref{li}) by using the following example. 
\begin{example}
Let V be a 4-dimensional vector space over $\mathbb{C}$. Taking the standard basis, we identify $V\cong \mathbb{C}^4$. Let 
\begin{equation*}
\mathbb{A}=(A_0,A_1,A_2,A_3)=\left(
\begin{pmatrix} 
1&0&0&0 \\
0&0&0&0 \\
0&0&0&0 \\
0&0&0&0
\end{pmatrix},
\begin{pmatrix}
0&0&0 \\
1&0&0 \\
0&0&0 \\
0&0&0
\end{pmatrix},
\begin{pmatrix}
0&0 \\
0&0 \\
1&0 \\
0&0
\end{pmatrix},
\begin{pmatrix}
0 \\
0 \\
0 \\
1
\end{pmatrix}
  \right),
\end{equation*}
\begin{equation*}
\mathbb{B}(t)=(B_0(t),B_1(t))=\left(
\begin{pmatrix} 
1&0&0&0 \\
0&t&0&0 \\
0&0&0&0 \\
0&0&0&0
\end{pmatrix},
\begin{pmatrix}
0&0 \\
0&0 \\
1&0 \\
0&t
\end{pmatrix}
  \right),
\end{equation*}
and let $U_i$ be a neighborhood of $A_i$, ($i=0,1,2,3$).  
In the rule of (\ref{li}), $\mathbb{B}(t)$ converges to $\mathbb{A}$ when $t\rightarrow 0$. In terms of (\ref{nbd}), we want to have 
\begin{equation} \label{wa}
\mathbb{B}(t)\in U_{\mathbb{PA}}(U_0,U_1,U_2,U_3)
\end{equation}
when $t<<0$. In fact (\ref{wa}) holds by the following.
\begin{equation*}
V(\mathbb{B})_0\supseteq V(\mathbb{A})_0, \,\, \overline{B_0(t)|_{V(\mathbb{A})_0}}\in U_0 \text{ since }\Lim\overline{B_0(t)|_{V(\mathbb{A})_0}}=
\begin{pmatrix}
1&0&0&0 \\
0&0&0&0 \\
0&0&0&0 \\
0&0&0&0
\end{pmatrix} ,
\end{equation*}
\begin{equation*}
V(\mathbb{B})_0\supseteq V(\mathbb{A})_1, \,\,\overline{B_0(t)|_{V(\mathbb{A})_1}}\in U_1 \text{ since }\Lim\overline{B_0(t)|_{V(\mathbb{A})_1}}=
\begin{pmatrix}
0&0&0 \\
1&0&0 \\
0&0&0 \\
0&0&0
\end{pmatrix} ,
\end{equation*}
\begin{equation*}
V(\mathbb{B})_1\supseteq V(\mathbb{A})_2, \,\,\overline{B_1(t)|_{V(\mathbb{A})_2}}\in U_2 \text{ since }\Lim\overline{B_1(t)|_{V(\mathbb{A})_2}}=
\begin{pmatrix}
0&0 \\
0&0 \\
1&0 \\
0&0
\end{pmatrix} ,
\end{equation*}
\begin{equation*}
V(\mathbb{B})_1\supseteq V(\mathbb{A})_3, \,\,\overline{B_1(t)|_{V(\mathbb{A})_3}}\in U_3 \text{ since }\Lim\overline{B_1(t)|_{V(\mathbb{A})_3}}=
\begin{pmatrix}
0 \\
0 \\
0 \\
1
\end{pmatrix} .
\end{equation*}
\end{example}
In fact (\ref{nbd}) induces a topology on $\PM$ by the following lemma.
\begin{lemma}[\cite{S1} Lemma 3.2.]
The sets 
\begin{equation*}
\{U_{\mathbb{PA}}(U_0,\dots,U_m)|U_i \text{ is a neighborhood of } \overline{A_i}\,\,(0\le i\le m)\}
\end{equation*}
satisfy the axiom of a base of neighborhoods of $\mathbb{PA}$, and hence define a topology in $\PM$.
\end{lemma}
Moreover the following theorem holds.
\begin{theorem}[\cite{S1} Theorem 5.1., Proposition 3.9, 3.10.]
The set $\PM$ is compact, and $\PGL(V)$ is dense open in $\PM$.
\end{theorem}
Here we regard an element of $\PGL(V)$ as a one-term element of $\PM$, and $\PM$ is a compactification of $\PGL(V)$. 
\subsection{Monoid structure of $\mathbb{P}\widetilde{M}$}
For $\mathbb{A}=(A_0,A_1,\dots,A_m)$, $\mathbb{B}=(B_0,B_1,\dots,B_n)$$\in\mathbb{P}\widetilde{M}$, define $\mathbb{AB}$ by removing the redundant matrices from
\begin{equation} \label{prod}
\begin{split}
\mathbb{AB}=(A_0B_0,A_1B_0,\dots,&A_mB_0,A_0B_1, \\
&\dots,A_mB_1,\dots,A_0B_n,\dots,A_mB_n).
\end{split}
\end{equation}
This defines a monoid structure on $\mathbb{P}\widetilde{M}$ (\cite{S1} Proposition 6.6.).
\section{$\mathbb{P}M$-monoids and braid $\PM$-monoids}
We shall define the $\mathbb{P}M$-monoid denoted by $\mathscr{R}_n$. This is motivated by the linear algebraic monoid theory.
\subsection{Motivation : Linear algebraic monoids}
Let $K$ be an algebraically closed field. Let $M_n=M_n(K)$ denote the set of all $n\times n$ matrices over $K$.
\begin{definition}
A linear algebraic monod is a submonoid of $M_n$ which is a Zariski closed subset.
\end{definition}
Let $M$ be a reductive monoid, i.e., $M$ is a linear algebraic monoid which is irreducible as algebraic set and has a connected reductive group of units. Let $T$ be a maximal torus of $G$. Then 
\begin{equation*}
R=\overline{N_{G}(T)}/T
\end{equation*}
is called a Renner monoid (\cite{Pu} Definition 11.2), where the closure is taken in Zariski topology. This contains the Weyl group $W=N_G(T)/T$ of $G$. Renner monoids play the central role in linear algebraic monoid theory like Weyl groups do in linear algebraic group theory, and have the following properties. Let $E(M)$ be the set of idempotents of $M$, and $P(e)=\{g\in G\mid ge=ege\}$ for $e\in E(M)$. Let $B$ be a Borel subgroup containing $T$, and $\Lambda(B)=\{e\in E(\overline{T})\mid P(e)\supseteq B\}$. We shall define an inverse monoid.
\begin{definition}
An inverse monoid is a monoid $M$ such that, for each $x\in M$, there is a unique $y\in M$ such that
\begin{equation*}
xyx=x, \,\, {\rm{and}} \,\, yxy=y.
\end{equation*}
\end{definition}
\begin{theorem}[\cite{So1} Theorem 5.10.] \label{thh}
Let $M$ be a reductive monoid, $e\in \Lambda(B)$. Then 
\begin{itemize}
\item[{\rm{(1)}}]$R$ is a finite inverse monoid.
\item[{\rm{(2)}}]The group of units of $R$ is $W$, and $R=WE(R)$.
\item[{\rm{(3)}}]$E(R)\simeq E(\overline{T})$.
\item[{\rm{(4)}}]$M=\displaystyle \sum_{\rho\in R}B\rho B$, and $B\rho B=B\rho^{\prime}B \,\, {\rm{implies}} \,\, \rho=\rho^{\prime}$. 
\item[{\rm{(5)}}]If $s\in S$ is a Coxeter generator, then $BsB \cdot B\rho B\subseteq Bs\rho B\cup B\rho B$.
\item[{\rm{(6)}}]$GeG=\displaystyle \sum_{\rho\in WeW}B\rho B$.
\item[{\rm{(7)}}]If $w_0\in W$ is the opposite element, then $Bw_0eB$ is open and dense in $GeG$.
\end{itemize}
\end{theorem}
The definition of $\mathbb{P}M$-monoid $\mathscr{R}_n$ is similar to the Renner monoid $R$.

\subsection{$\mathbb{P}M$-monoids and those structures}
Let $T$ be a maximal torus of $\PGL_n$. Then we consider the following monoid 
\begin{equation*}
\mathscr{R}_n=\overline{N_{\PGL_n}(T)}/T,
\end{equation*}
where the closure is taken in the topology of $\mathbb{P}M$. We call this monoid $\mathscr{R}_n$ a $\mathbb{P}M$-monoid. We next consider the structure of a $\PM$-monoid. Actually, the structure of a $\PM$-monoid can be described in terms of matched pairs. We first explain the matched pairs (cf. \cite{MS},\cite{T}). Let $S$ be a monoid. We denote the unit element of $S$ by $1_S$.
\begin{definition}
Let $S, B$ be monoids which have binary operations $\rightharpoonup:S\times B\rightarrow B$ and $\leftharpoonup:S\times B\rightarrow S$. A matched pair of monoids means a triple $(S,B,\sigma)$, where $S,B$ are monoids and
\begin{equation*}
\sigma:S\times B\rightarrow B\times S, \, (s,b)\mapsto (s\rightharpoonup b,s\leftharpoonup b)
\end{equation*}
is a map satisfying the following conditions : 
\begin{enumerate}
\item[(1)] \label{m1}
$s\rightharpoonup(t\rightharpoonup b)=st\rightharpoonup b$,
\item[(2)] \label{m2}
$st\leftharpoonup b=(s\leftharpoonup(t\rightharpoonup b))(t\leftharpoonup b)$,
\item[(3)] \label{m3}
$(s\leftharpoonup b)\leftharpoonup c=s\leftharpoonup bc$,
\item[(4)] \label{m4}
$s\rightharpoonup bc=(s\rightharpoonup b)((s\leftharpoonup b)\rightharpoonup c)$,
\item[(5)] \label{m5}
$1_S\rightharpoonup b=b$,
\item[(6)] \label{m6}
$s\rightharpoonup 1_B=1_B$,
\item[(7)] \label{m7}
$s\leftharpoonup 1_B=s$,
\item[(8)] \label{m8}
$1_S\leftharpoonup b=1_S$
\end{enumerate}
for $s,t\in S$, $b,c\in B$.
\end{definition}
The product $B\times S$ forms a monoid with product
\begin{equation*}
(b,s)(c,t)=(b(s\rightharpoonup c),(s\leftharpoonup c)t).
\end{equation*}
This monoid is denoted by $B\Join_{\sigma} S$. \\
Let
\begin{equation*}
\begin{split}
P_n=\Biggl\{(\{i_1,\dots,i_{k_1}\},\{i_{k_1+1},\dots,i_{k_2}\},&\dots,\{i_{k_m+1},\dots,i_n\}) \\
&\mid \begin{matrix}\{i_1\dots,i_n\}=\{1,\dots,n\} \\ 1\le k_1<k_2<\dots<k_{m-1}< n\end{matrix} \Biggr\}.
\end{split}
\end{equation*}
The element of $P_n$ is called an ordered set partitions of $[n]:=\{1,2,\dots,n\}$. The set $P_n$ has a monoid structure defined by
\begin{equation*}
(p_1,\dots,p_m)*(p_1^{\prime},\dots,p^{\prime}_{m^{\prime}}):=(p_1\cap p_1^{\prime},\dots,p_m\cap p_1^{\prime},\dots,p_1\cap p^{\prime}_{m^{\prime}},\dots,p_m\cap p^{\prime}_{m^{\prime}}).
\end{equation*}
Then the following proposition holds.
\begin{proposition} \label{rm}
Let $\mathscr{R}_n$ be the $\PM$-monoid, $S_n$ the symmetric group and $P_n$ the collection of the ordered set partitions of $[n]$. Define a map
\begin{equation*}
\varphi:P_n\times S_n\rightarrow S_n\times P_n,\, ((p_1,\dots,p_m),w)\mapsto (w,(w^{-1}(p_1),\dots,w^{-1}(p_m))).
\end{equation*}
Then
\begin{equation*}
\mathscr{R}_n\simeq S_n\Join_{\varphi} P_n.
\end{equation*}
\end{proposition}
\begin{proof}
Since $N_{\PGL_n}(T)=\{\displaystyle \sum_{j=1}^n t_jE_{\pi(j)j}\mid t_j\in\mathbb{C}^*,\pi\in S_n\}$, we have
\begin{equation*}
\overline{N_{\PGL_n}(T)}=\left\{\left(\displaystyle \sum_{j\in p_1}t_jE_{\pi(j)j},\dots,\displaystyle \sum_{j\in p_m}t_jE_{\pi(j)j}\right)\mid 
\begin{matrix}
t_j\in\mathbb{C}^*, \\
(p_1,\dots,p_m)\in P_n,\pi\in S_n
\end{matrix}
\right\}.
\end{equation*}
Thus
\begin{equation*}
\overline{N_{\PGL_n}(T)}/T=\left\{\left(\displaystyle \sum_{j\in p_1}E_{\pi(j)j},\dots,\displaystyle \sum_{j\in p_m}E_{\pi(j)j}\right)\mid (p_1,\dots,p_m)\in P_n,\pi\in S_n\right\}.
\end{equation*}
Then we have the following bijective correspondence as sets.
\begin{equation} \label{mon}
\begin{split}
\overline{N_{\PGL_n}(T)}/T&\simeq S_n\times P_n: \\
\left(\displaystyle \sum_{j\in p_1}E_{\pi(j)j},\dots,\displaystyle \sum_{j\in p_m}E_{\pi(j)j}\right)&\mapsto (\pi,(p_1,\dots,p_m)).
\end{split}
\end{equation}
To introduce a monoid structure on $S_n\times P_n$, we recall a monoid structure of $\overline{N_{\PGL_n}(T)}/T$ (cf. (\ref{prod})).
\begin{equation*}
\begin{split}
&\left(\displaystyle \sum_{j\in p_1}E_{\sigma(j)j},\dots,\displaystyle \sum_{j\in p_m}E_{\sigma(j)j}\right)\cdot\left(\displaystyle \sum_{k\in p_1^{\prime}}E_{\sigma^{\prime}(k)k},\dots,\displaystyle \sum_{k\in p_n^{\prime}}E_{\sigma^{\prime}(k)k}\right) \\
&=\left(\displaystyle \sum_{j\in p_1}E_{\sigma(j)j}\displaystyle \sum_{k\in p_1^{\prime}}E_{\sigma(k)k},\displaystyle \sum_{j\in p_2}E_{\sigma(j)j}\displaystyle \sum_{k\in p_1^{\prime}}E_{\sigma(k)k},\dots,\displaystyle \sum_{j\in p_m}E_{\sigma^{\prime}(j)j}\displaystyle \sum_{k\in p_n^{\prime}}E_{\sigma^{\prime}(k)k}\right)\\
&=\left(\displaystyle \sum_{l\in {\sigma^{\prime}}^{-1}(p_1)\cap p_1^{\prime}}E_{\sigma\sigma^{\prime}(l)l},\displaystyle \sum_{l\in {\sigma^{\prime}}^{-1}(p_2)\cap p_1^{\prime}}E_{\sigma\sigma^{\prime}(l)l},\dots,\displaystyle \sum_{l\in {\sigma^{\prime}}^{-1}(p_m)\cap p_n^{\prime}}E_{\sigma\sigma^{\prime}(l)l}\right).
\end{split}
\end{equation*}
By the above calculation, we define a product on $S_n\times P_n$
\begin{equation} \label{pro}
\begin{split}
(\sigma,(p_1,\dots,p_m))\cdot&(\sigma^{\prime},(p_1^{\prime},\dots,p_n^{\prime})) \\
&:=(\sigma\sigma^{\prime},({\sigma^{\prime}}^{-1}(p_1),\dots,{\sigma^{\prime}}^{-1}(p_m))*(p_1^{\prime},\dots,p_n^{\prime})).
\end{split}
\end{equation}
Then (\ref{mon}) becomes the isomorphism of monoids. On the other hand, we define a map
\begin{equation*}
\varphi:P_n\times S_n\rightarrow S_n\times P_n:((p_1,\dots,p_m),\sigma)\mapsto (\sigma, (\sigma^{-1}(p_1),\dots,\sigma^{-1}(p_m))).
\end{equation*}
Then $(P_n,S_n,\varphi)$ satisfies (1)-(8) in Definition \ref{m1}, and becomes a matched pair. The monoid structure of $S_n\Join_\varphi P_n$ coincides with (\ref{pro}). Therefore
\begin{equation*}
\mathscr{R}_n\simeq S_n\Join_{\varphi} P_n
\end{equation*} 
as monoids.
\end{proof}

\subsection{Properties of $\mathbb{P}M$-monoids}
A $\mathbb{P}M$-monoid has the following properties analogous to Theorem \ref{thh}.
\begin{proposition}
Let $\mathscr{R}_n=\overline{N_{\PGL_n}(T)}/T$, and
\begin{equation*}
\begin{split}
\Lambda_n=\Biggl\{(\displaystyle\sum_{j\in p_1}E_{jj},&\displaystyle\sum_{j\in p_2}E_{jj},\dots,\displaystyle\sum_{j\in p_n}E_{jj}), \\
&\mid
\begin{matrix} (p_1,\dots,p_n)=(\{1,\dots,k_1\},\dots,\{k_{m-1}+1,\dots,n\}) \\ 1\le k_1<k_2<\dots<k_{m-1}< n
\end{matrix}\Biggr\}
\end{split}
\end{equation*}
\begin{itemize}
\item[{\rm{(a)}}]$\mathscr{R}_n$ is a finite inverse monoid. Moreover the number of  its elements is 
\begin{equation} \label{rnu}
\begin{split}
\,\,&|\mathscr{R}_n| \\
&=\displaystyle\sum_{r_1+\dots+r_m=n}{\begin{pmatrix}n \\ r_1\end{pmatrix}}^2r_1!{\begin{pmatrix}n-r_1 \\ r_2\end{pmatrix}}^2r_2!\dots{\begin{pmatrix}n-r_1-\dots-r_{n-1} \\ r_n\end{pmatrix}}^2r_n! \\
&=n!\displaystyle \sum_{m=1}^n m!S(n,m)
\end{split}
\end{equation}
where $S(n,m)$ is the Stirling numbers of the second kind, i.e., $S(n,m)$  is the number of ways of partitioning a set of $n$ elements into $m$ non-empty subsets.
\item[{\rm{(b)}}]The unit group of $\mathscr{R}_n$ is $W:=N_{\PGL}(T)/T$ and $\mathscr{R}_n=WE(\mathscr{R}_n)$.
\item[{\rm{(c)}}]$E(\mathscr{R}_n)=\displaystyle\bigcup_{w\in W}w\Lambda_nw^{-1}$.
\item[{\rm{(d)}}]$\mathscr{R}_n=\displaystyle\bigsqcup_{e\in \Lambda_n}WeW$.
\end{itemize}
\end{proposition}
\begin{proof}
(a) For any $(\sigma,p)\in \mathscr{R}_n$, let $(\sigma,p)^{*}:=(\sigma^{-1},\sigma(p))$. Then
\begin{equation*}
(\sigma,p)=(\sigma,p)(\sigma,p)^{*}(\sigma,p),\,\,(\sigma,p)^{*}=(\sigma,p)^{*}(\sigma,p)(\sigma,p)^{*}.
\end{equation*}
Thus $\mathscr{R}_n$ is an inverse monoid. We next consider the number $|\mathscr{R}_n|$. We fix a partition $r_1+r_2+\dots+r_m=n$. We first choose $r_1$ columns and $r_1$ rows among the $n$ columns and $n$ rows, and choose a placement of  $1$'s in the place of $r_1\times r_1$ permutation matrices. Next we choose $r_2$ columns and $r_2$ rows among the $n-r_1$ columns and $n-r_1$ rows, and choose a placement of  $1$'s in the place of $r_2\times r_2$ permutation matrices. We repeat this process and sum up over partitions $r_1+r_2+\dots+r_m=n$, and then we obtain the first equality of (\ref{rnu}). 

On the other hand, let $P_{n,m}$ be the collection of ordered set partitions of $[n]$ with $m$ blocks. Then $|P_{n,m}|=m!S(n,m)$ by the definition of the Stirling numbers of the second kind. Thus we obtain the second equality of (\ref{rnu}) since $|P_n|=\displaystyle\sum_{m=1}^n|P_{n,m}|$ and $|\mathscr{R}_n|=|S_n||P_n|$. 

(b) First we see 
\begin{equation} \label{idem}
E(\mathscr{R}_n)=\left\{(\displaystyle\sum_{j\in p_1}E_{jj},\displaystyle\sum_{j\in p_2}E_{jj},\dots,\displaystyle\sum_{j\in p_n}E_{jj})\mid
 (p_1,\dots,p_n)\in P_n\right\}.
\end{equation}
In fact, if $(\sigma,p)(\sigma,p)=(\sigma,p)$ for $(\sigma,p)\in \mathscr{R}_n$, then $\sigma^2=\sigma$, i.e., $\sigma=e$. Then $\mathscr{R}_n=WE(\mathscr{R}_n)$.

(c) follows from (\ref{idem})

(d)\begin{equation*}
\begin{split}
&\displaystyle\bigsqcup_{e\in \Lambda_n}WeW \\
&=\left\{\sigma(\displaystyle\sum_{j\in p_1}E_{jj},\dots,\displaystyle\sum_{j\in p_n}E_{jj})\tau\mid
\begin{matrix} (p_1,\dots,p_n)=(k_1,\dots,k_{m-1}) \\ \sigma,\tau\in W
\end{matrix}\right\} \\
&=\left\{(\displaystyle\sum_{j\in p_1}E_{\sigma^{-1}(j)\tau(j)},\dots,\displaystyle\sum_{j\in p_n}E_{\sigma^{-1}(j)\tau(j)})\mid
\begin{matrix} (p_1,\dots,p_n)=(k_1,\dots,k_{m-1}) \\ \sigma,\tau\in W
\end{matrix}\right\} \\
&=\left\{(\displaystyle\sum_{k\in \tau(p_1)}E_{(\tau\sigma)^{-1}(k)k},\dots,\displaystyle\sum_{k\in \tau(p_n)}E_{(\sigma\tau)^{-1}(k)k})\mid
\begin{matrix} (p_1,\dots,p_n)=(k_1,\dots,k_{m-1}) \\ \sigma,\tau\in W
\end{matrix}\right\} \\
&=\mathscr{R}_n,
\end{split}
\end{equation*}
where we denote
\begin{equation} \label{par}
(k_1,\dots,k_{m-1})=(\{1,\dots,k_1\},\dots,\{k_{m-1}+1,\dots,n\}).
\end{equation}
\end{proof}
\subsection{Presentation of $\PM$-monoids}
\subsubsection{Presentation of monoids}
Let $X$ be an alphabet, i.e., a set whose elements are called letters, and denote by $X^*$ the free monoid on $X$. For $R\subseteq X^{*}\times X^{*}$, let $R^{\#}$ denote the smallest congruence on $X^*$ containing $R$. We say a monoid $M$ has a presentation $\langle X\mid R\rangle$ if $M\simeq X^*/R^{\#}$. An element $(w_1,w_2)\in R$ is called a relation and written as $w_1=w_2$. 
\subsubsection{Presentation of rook monoids}
Let $R_n$ be a set of $n\times n$ zero-one matrices which have at most one entry equal to 1 in each row and in each column. The monoid $R_n$ is called the rook monoid. The rook monoid $R_n$ has the following presentation using generating set and relations:
\begin{theorem}[\cite{Go2} Prop 1.6.]
The rook monoid has a monoid presentation with generating set $\{s_1,\dots,s_{n-1},e_0,\dots,e_{n-1}\}$ and defining relations:
\begin{equation*}
\begin{aligned}
s_i^2&=1\,\,&(1&\le i\le n-1), \\
s_is_j&=s_js_i\,\,&(1&\le i,j\le n-1 ,\, |i-j|\ge 2), \\
s_is_{i+1}s_i&=s_{i+1}s_is_{i+1}\,\,&(1&\le i\le n-1), \\
e_ie_j&=e_je_i=e_{{\rm{min}}(i,j)}\,\, &(0&\le i,j\le n-1), \\
e_js_i&=s_ie_j \,\, &(1&\le i<j\le n-1), \\
e_js_i&=s_ie_j=e_j\,\, &(0&\le j<i\le n-1), \\
e_is_ie_i&=s_ie_{i-1}\,\,&(1&\le i\le n-1).
\end{aligned}
\end{equation*}
\end{theorem} 

\subsubsection{Presentation of $\PM$-monoids}
We construct a presentation for $\mathscr{R}_n=\overline{N_{\PGL}(T)}/T$ like the rook monoid. We first define some notations. For $i=1,\dots,n-1$ and a partition $(k_1,\dots,k_{m-1})$ (cf. (\ref{par})), if there exists $j\in \{1,\dots,n\}$ such that $\{i,i+1\}\subseteq\{k_{j-1}+1,\dots,k_j\}$, then we set 
\begin{equation*}
i_{*}:=j.
\end{equation*} 
For $\sigma\in S_n$ we define a map $\varphi_{\sigma}:P_n\rightarrow P_n$ by
\begin{equation*}
(p_1,\dots,p_m)\mapsto (\sigma^{-1}(p_1),\dots,\sigma^{-1}(p_m)).
\end{equation*}
We define a set 
\begin{equation*}
\Pi_n=\{(k_1,\dots,k_{m-1}):1\le k_1<\dots<k_{m-1}< n\}
\end{equation*}
, where $(k_1,\dots,k_{m-1})$ is $(\ref{par})$.
For $p\in P_n$, take an element $w\in S_n$ such that $wpw^{-1}\in \Pi_n$, and set
\begin{equation*}
u^{w}(p):=wpw^{-1}\in \Pi_n.
\end{equation*}
We also set
\begin{equation*}
{\rm{Ad}}(\sigma)(e):=\sigma^{-1}e\sigma.
\end{equation*}
Using these notations we obtain the following monoid presentation of the $\PM$-monoid $\mathscr{R}_n$.
\begin{proposition} \label{rep}
The $\PM$-monoid $\mathscr{R}_n$ has a monoid presentation with generating set
\begin{equation*}
\{s_1,\dots,s_{n-1},e_{k_1,\dots,k_{m-1}}\,\, (1\le k_1<\dots<k_{m-1}< n)\}
\end{equation*}
and defining relations
\begin{align}
s_i^2&=1  &\,\,(&1\le i\le n-1),  \label{re1} \\
s_is_j&=s_js_i  &(&1\le i,j \le n-1,\,|i-j|\ge 2), \label{re2}   \\
s_is_{i+1}s_i&=s_{i+1}s_is_{i+1}  &(&1\le i\le n-1), \label{re3} 
\end{align}
\begin{align}
e_{k_1,\dots,k_{i_{*}},\dots,k_{m-1}}s_i&=s_ie_{k_1,\dots,k_{i_{*}},\dots,k_{m-1}} \label{re4}  \\ 
&\begin{pmatrix}
1\le i\le n-1 \\
1\le k_1<\dots<k_{i_{*}}<\dots<k_{m-1}< n
\end{pmatrix}, \notag \\
e_{k_1,\dots,k_{m-1}}s_{i_1}\dots s_{i_r}e_{l_1,\dots,l_{m^{\prime}-1}}
&={\rm{Ad}}(s_{j_1}\dots s_{j_t})(e_{q})s_{i_1}\dots s_{i_r}  \label{re5} \\
&\begin{pmatrix}
1\le k_1<\dots<k_{m-1}< n \\
1\le l_1<\dots<l_{m^{\prime}-1}< n \\
\{i_1,i_1+1\}\nsubseteq \{k_{l-1}+1,\dots,k_l\},  ^{\forall}l=1,\dots,n \\
q=u^{s_{j_1}\dots s_{j_t}}((k_1,\dots,k_{m-1})*\varphi_{(s_{i_1}\dots s_{i_r})^{-1}}((l_1,\dots,l_{m^{\prime}-1}))) \notag
\end{pmatrix}.
\end{align}
\end{proposition}
\begin{proof}
Let 
\begin{equation*}
e_{k_1,\dots,k_{m-1}}:=\left(\displaystyle\sum_{j=1}^{k_1}E_{jj},\displaystyle\sum_{j=k_1+1}^{k_2}E_{jj},\dots,\displaystyle\sum_{j=k_{m-1}+1}^{n}E_{jj}\right)
\end{equation*}
and $s_i=(i,i+1)$. These elements satisfy the above relations. Let $\mathscr{R}_n^{\prime}$ be the monoid generated by elements $s_1^{\prime},\dots,s_{n-1}^{\prime}, e_{k_1,\dots,k_{m-1}}^{\prime}$ subject to the defining relations (\ref{re1})-(\ref{re5}). Since $\mathscr{R}_n$ satisfies (\ref{re1})-(\ref{re5}), there is a surjective monoid homomorphism $\theta:\mathscr{R}_n^{\prime}\rightarrow \mathscr{R}_n$ such that $\theta(s_i^{\prime})=s_i$ and $\theta(e_{k_1,\dots,k_{m-1}}^{\prime})=e_{k_1,\dots,k_{m-1}}$. Let $S_n^{\prime}=\langle s_1^{\prime},\dots,s_{n-1}^{\prime} \rangle\subseteq \mathscr{R}_n^{\prime}$. To show that $\theta:\mathscr{R}_n^{\prime}\rightarrow \mathscr{R}_n$ is an isomorphism of monoids, it suffices to show that $|\mathscr{R}_n^{\prime}|\le|\mathscr{R}_n|$, where $|\mathscr{R}_n|$ is given by (\ref{rnu}). \\
We consider the following set 
\begin{equation} \label{ses}
\displaystyle\bigcup_{1\le k_1<\dots< k_m< n}S_n^{\prime}e_{k_1,\dots,k_{m-1}}^{\prime}S_n^{\prime}.
\end{equation}
Using relations (\ref{re4}), (\ref{re5}), we can show that the set (\ref{ses}) is stable under the left multiplication by $e_{k_1,\dots,k_{m-1}}^{\prime}$ and $S_n^{\prime}$. The set (\ref{ses}) contains $e_{n}=1$. Thus the set (\ref{ses}) contains $\mathscr{R}_n$. Therefore we have
\begin{equation*}
\mathscr{R}_n^{\prime}=\displaystyle\bigcup_{1\le k_1<\dots< k_m< n}S_n^{\prime}e_{k_1,\dots,k_{m-1}}^{\prime}S_n^{\prime}.
\end{equation*}
We fix $1\le k_1<\dots<k_{m-1}< n$ and let 
\begin{equation*}
\begin{split}
S_{k_1,\dots,k_{m-1}}^{\prime}&:=\langle s_1^{\prime},\dots,s_{k_1-1}^{\prime},s_{k_1+1}^{\prime},\dots,s_{k_2-1}^{\prime},s_{k_2+1}^{\prime},\dots,s_{k_{m-1}-1}^{\prime},s_{k_{m-1}+1}^{\prime},\dots,s_{n-1}^{\prime} \rangle \\
&\simeq S_{k_1}\times S_{k_2-k_1}\times \dots \times S_{n-k_{m-1}}.
\end{split}
\end{equation*}
Write $S_n^{\prime}=S_{k_1,\dots,k_{m-1}}^{\prime}X_{k_1,\dots,k_{m-1}}$, where $X_{k_1,\dots,k_{m-1}}$ is a set of coset representatives. Then by the relation (\ref{re4}) of the above relations,
\begin{equation*}
\begin{split}
e_{k_1,\dots,k_{m-1}}^{\prime}S_n^{\prime} 
&=e_{k_1,\dots,k_{m-1}}^{\prime}S_{k_1,\dots,k_{m-1}}^{\prime}X_{k_1,\dots,k_{m-1}}\\
&=S_{k_1,\dots,k_{m-1}}^{\prime}e_{k_1,\dots,k_{m-1}}^{\prime}X_{k_1,\dots,k_{m-1}}\\
&\subseteq S_n^{\prime}e_{k_1,\dots,k_{m-1}}^{\prime}X_{k_1,\dots,k_{m-1}}.
\end{split}
\end{equation*}
Thus
\begin{equation*}
\begin{split}
|S_n^{\prime}e_{k_1,\dots,k_{m-1}}^{\prime}S_n^{\prime}| 
&\le |S_n^{\prime}e_{k_1,\dots,k_{m-1}}^{\prime}X_{k_1,\dots,k_{m-1}}| \\
&\le |S_n^{\prime}e_{k_1,\dots,k_{m-1}}^{\prime}||X_{k_1,\dots,k_{m-1}}| \\
&=\frac{n!}{k_1!(k_2-k_1)!\dots(n-k_{m-1})!}|S_n^{\prime}e_{k_1,\dots,k_{m-1}}|\\
&\le\frac{(n!)^2}{k_1!(k_2-k_1)!\dots(n-k_{m-1})!}.
\end{split}
\end{equation*}
Therefore
\begin{equation*}
\begin{split}
|\displaystyle\bigcup_{1\le k_1<\dots< k_m< n}S_n^{\prime}e_{k_1,\dots,k_{m-1}}^{\prime}S_n^{\prime}| 
&\le\displaystyle\sum_{1\le k_1<\dots< k_m< n}\frac{(n!)^2}{k_1!(k_2-k_1)!\dots(n-k_{m-1})!}\\
&=\displaystyle\sum_{r_1+\dots+r_m=n}\frac{(n!)^2}{r_1!r_2!\dots r_m!}.
\end{split}
\end{equation*}
On the other hand 
\begin{equation*}
\begin{split}
&\displaystyle\sum_{r_1+\dots+r_m=n}{\begin{pmatrix}n \\ r_1\end{pmatrix}}^2r_1!{\begin{pmatrix}n-r_1 \\ r_2\end{pmatrix}}^2r_2!\dots{\begin{pmatrix}n-r_1-\dots-r_{n-1} \\ r_n\end{pmatrix}}^2r_n! \\
&=\displaystyle\sum_{r_1+\dots+r_m=n}\frac{(n!)^2}{r_1!r_2!\dots r_m!}.
\end{split}
\end{equation*}
\end{proof}

\begin{remark}
In the relation (\ref{re5}), the adjustment by ${\rm{Ad}}$ is needed since the right hand of (\ref{re5}) is not necessarily of the form $e_{k_1,\dots,k_{m-1}}s_{i_1}\dots s_{i_r}$. For example, in $\mathscr{R}_3$
\begin{equation*}
e_2(s_2s_1s_2)e_1=s_1s_2e_1s_2s_1(s_2s_1s_2).
\end{equation*}
\end{remark}

\subsection{Braid groups and inverse braid monoids}
We review the Artin braid group $B_n$ \cite{A1}.
\begin{definition}
The braid group $B_n$ is the group generated by $n-1$ elements $\sigma_1,\sigma_2,\dots,\sigma_{n-1}$ with the braid relations 
\begin{align*}
\sigma_i\sigma_j&=\sigma_j\sigma_i \,\, &(&i=1,2,\dots,n-1, |i-j|\ge 2), \\
\sigma_i\sigma_{i+1}\sigma_i&=\sigma_{i+1}\sigma_i\sigma_{i+1} \,\, &(&i=1,2,\dots,n-2).
\end{align*}
\end{definition}
An element of the braid group can be represented by the braid diagram. 

The inverse braid monoid $IB_n$ was constructed by D. Easdown and T. G. Lavers \cite{EL}. It arises from an operation of braids : deleting one or several strings. An element of $IB_n$ is called a partial braid, and a product of two partial braids is defined (Section $1$ of \cite{EL}). Thus $IB_n$ is the monoid with product of partial braids.
\subsection{Braid $\PM$-monoids}
We define a braid monoid according to Proposition \ref{rep}. The notations are the same as those in Proposition \ref{rep}, and we add the following notation. We denote by $b|_I$ an element of  braid group of  $\#I$-strings for $b\in B_n$ and $I\subset\{1,\dots,n\}$.  If $s_{i_1},\dots,s_{i_r}\in B_n$ satisfy $s_{i_1}\dots s_{i_r}|_I=id|_I$, where $I\subset\{1,\dots,n\}$ and $id$ is the identity braid in $B_n$, then we abbreviate this condition as $\{i_1,\dots,i_r\}|_I=id$.
\begin{definition} \label{Bpm}
The braid $\mathbb{P}M$-monoid is a monoid which is defined by the monoid presentation with generating set 
\begin{equation*}
\{s_1^{\pm 1},\dots,s_{n-1}^{\pm 1},e_{k_1,\dots,k_{m-1}}\,\,(1\le k_1<\dots<k_{m-1}< n)\}
\end{equation*}
and defining relations
\begin{align} 
s_is_i^{-1}&=s_i^{-1}s_i=1 \,\, &(&1\le i\le n-1), \label{re1-} \\
s_is_j&=s_js_i ,\,\ &(&1\le i,j \le n-1,\,|i-j|\ge 2), \label{re2-} \\
s_is_{i+1}s_i&=s_{i+1}s_is_{i+1} \,\, &(&1\le i\le n-1), \label{re3-}
\end{align}
\begin{align}
s_{i_1}^{\pm 1}\dots s_{i_r}^{\pm 1}e_{k_1,\dots,k_{m-1}}s_{j_1}^{\pm 1}\dots s_{j_t}^{\pm 1}&=e_{k_1,\dots,k_{m-1}}   \label{re4-} \\
&\begin{pmatrix}
\{i_1,\dots,i_r,j_1,\dots,j_t\}|_{\{k_{j-1}-1,\dots,k_j\}}=id \\ ^{\forall}j=1,\dots ,m 
\end{pmatrix}, \notag \\
e_{k_1,\dots,k_{m-1}}s_{i_1}^{\pm 1}\dots s_{i_r}^{\pm 1}e_{l_1,\dots,l_{m^{\prime}-1}} 
&={\rm{Ad}}(s_{j_1}^{\pm 1}\dots s_{j_t}^{\pm 1})(e_{q})s_{i_1}^{\pm 1}\dots s_{i_r}^{\pm 1} \label{re5-} \\
&\begin{pmatrix} 
1 \le k_1<\dots<k_{m-1}< n \\
1\le l_1<\dots<l_{m^{\prime}-1}< n  \\
\{i_1,i_1+1\}\nsubseteq \{k_{l-1}+1,\dots,k_l\}, ^{\forall} l=1,\dots,m \\
q=u^{s_{j_1}^{\pm 1}\dots s_{j_t}^{\pm 1}}((k_1,\dots,k_{m-1})*\varphi_{(s_{i_1}\dots s_{i_r})^{-1}}((l_1,\dots,l_{m^{\prime}-1})))
\end{pmatrix}. \notag
\end{align}
\end{definition}

\section{Main results}
\subsection{Braid diagram of the braid $\PM$-monoids}
We denote by $\mathscr{M}$ the monoid defined in Definition \ref{Bpm}. To describe the monoid $\mathscr{M}$ geometrically we shall define a $\mathbb{P}M$-braid.  

First, we shall define an arc.
\begin{definition}
An arc is the image of an embedding from the unit interval $[0,1]$ into $\mathbb{R}^3$. 
\end{definition}

Take the usual coordinate system $(x,y,z)$ for $\mathbb{R}^3$. Choose $z_0^{(m)}<z_1^{(m)}<\dots<z_0^{(2)}<z_1^{(2)}<z_0^{(1)}<z_1^{(1)}$. Mark $n\geq 0$ distinct points $P_1^i,\dots,P_n^i$ on a line in the plane $z=z_1^{(i)}$, and project this orthogonally on the plane $z=z_0^{(i)}$,  yielding points $Q_1^i,\dots,Q_n^i$ for each $i=1,\dots,m$. 

A $\mathbb{P}M$-braid on $n$ strings is a system
\begin{equation*}
\beta=\{\beta_1,\dots,\beta_{k_1},\beta_{k_1+1},\dots,\beta_{k_2},\beta_{k_2+1},\dots,\beta_{k_{m-1}+1},\dots,\beta_n\}
\end{equation*}
of $n$ arcs for some $1\le k_1<k_2<\dots<k_{m-1}<n$ such that 
\begin{enumerate}
\item[(1)] There is a partial one-one mapping of rank $k_1$ 
\begin{equation*}
\Phi_1^{\beta}:\{1,\dots,n\}\rightarrow \{1,\dots,n\}
\end{equation*}
with domain $\{i_1,\dots,i_{k_1}\}$ such that $\beta_j$ connects $P_{i_j}^1$ to $Q_{\Phi_1^{\beta}(i_j)}^1$ for $j=1,\dots,k_1$.\\
There is a partial one-one mapping of rank $k_2-k_1$ 
\begin{equation*}
\Phi_2^{\beta}:\{1,\dots,n\}\backslash\{i_1,\dots,i_{k_1}\}\rightarrow \{1,\dots,n\}\backslash\{i_1,\dots,i_{k_1}\}
\end{equation*}
with domain $\{i_{k_1+1},\dots,i_{k_2}\}$ such that $\beta_j$ connects $P_{i_j}^2$ to $Q_{\Phi_1^{\beta}(i_j)}^2$ for $j=k_1+1,\dots,k_2$.
\begin{center}
$\dots\dots$
\end{center}
There is a partial one-one mapping of rank $n-k_{m-1}$
\begin{equation*}
\Phi_m^{\beta}:\{1,\dots,n\}\backslash\{i_1,\dots,i_{k_{m-1}}\}\rightarrow \{1,\dots,n\}\backslash\{i_1,\dots,i_{k_{m-1}}\}
\end{equation*}
with domain $\{i_{k_{m-1}+1},\dots,i_n\}$ such that $\beta_j$ connects $P_{i_j}^m$ to $Q_{\Phi_m^{\beta}(i_j)}^m$ for $j=k_{m-1}+1,\dots,n$. \\
\item[(2)] For $j=1,\dots,m$, the arc $\beta_l$ intersects the plane $z=z_0^{(j)}$ exactly once, and $\beta_l$ intersects the plane $z=z_1^{(j)}$ exactly once, for $l=k_{j-1}+1,\dots,k_j$, and $\beta_s$ does not intersect $z=z_0^{(t)}$, $z=z_1^{(t)}$ for $s\neq t$. \\
\item[(3)] For $j=1,\dots,m$ the union $\beta_{k_{j-1}+1}\cup\dots\cup\beta_{k_j}$ of the arcs intersects each parallel plane $z=z_0^{(j)},z=z_1^{(j)}$ at exactly $k_j-k_{j-1}$ distinct points. 
\end{enumerate}
\begin{example}
The following is a $\PM$-braid. \\
\raisebox{50pt}{$\beta=$}
\includegraphics[width=4cm,bb=0 0 320 300]{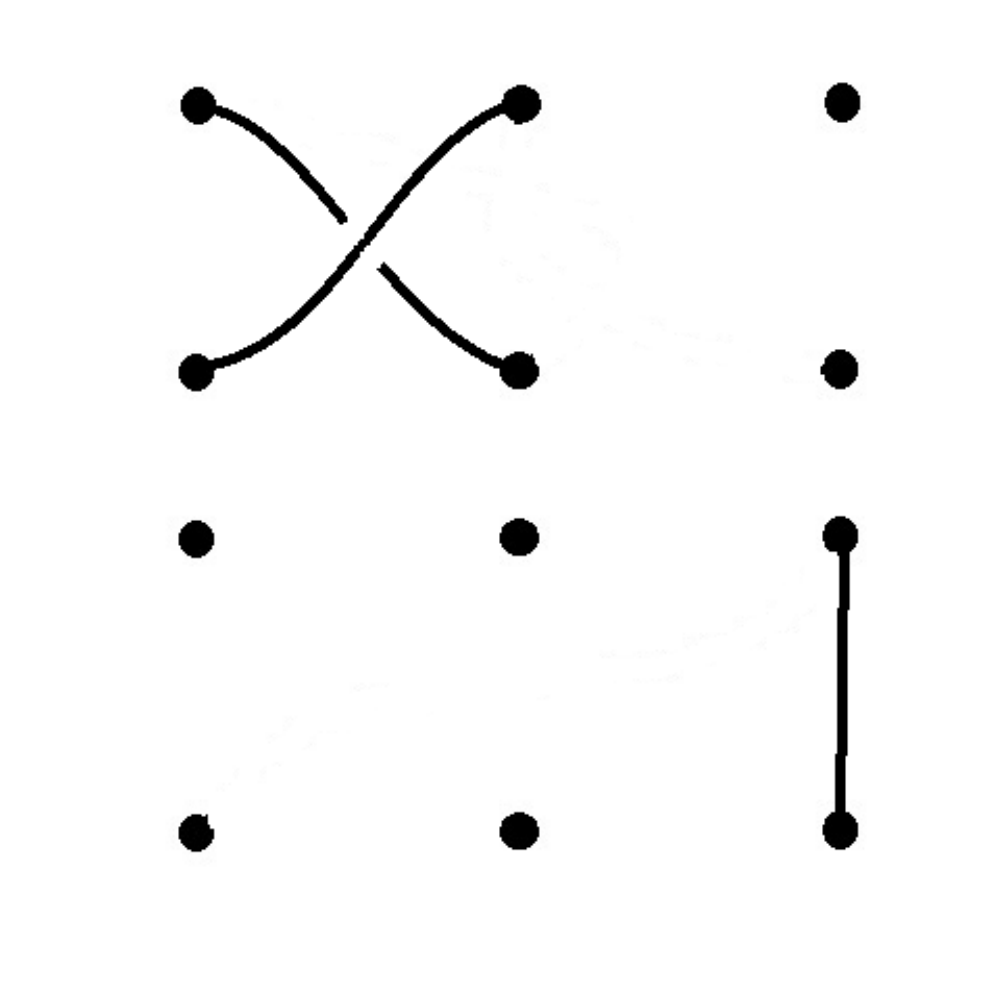} 
\end{example}

Two $\mathbb{P}M$-braids
\begin{equation*}
\begin{split}
 &\beta=\{\beta_1,\dots,\beta_{k_1},\beta_{k_1+1},\dots,\beta_{k_2},\beta_{k_2+1},\dots,\beta_{k_{m-1}+1},\dots,\beta_n\}, \\
 &\gamma=\{\gamma_1,\dots,\gamma_{k_1},\gamma_{k_1+1},\dots,\gamma_{k_2},\gamma_{k_2+1},\dots,\gamma_{k_{m^{\prime}-1}+1},\dots,\gamma_n\}
 \end{split}
 \end{equation*}
are defined to be equivalent if \\
\begin{itemize}
\item[(1)] $m=m^{\prime}$ and $\Phi_i^{\beta}=\Phi_i^{\gamma}$ for $i=1,\dots,m$, \\
\item[(2)] $\beta$ and $\gamma$ are homotopy equivalent, i.e., there exist continuous maps 
\begin{equation*}
F_j:[0,1]\times [0,1]\rightarrow \mathbb{R}^3,\,\,\,\,\,\,\,\,\,\,\, (j=1,\dots,m)
\end{equation*}
such that for all $s, t\in [0,1]$,
\begin{align*}
&\begin{matrix}
F_j(t,0)&=&\beta_j(t)  \\
F_j(t,1)&=&\gamma_j(t)
\end{matrix}
&(&j=1,\dots ,m), \\
&\begin{matrix}
F_j(0,s)&=&P_{i_j}^1  \\
F_j(1,s)&=&Q_{\Phi_1^{\beta}(i_j)}^1
\end{matrix}
&(&j=1,\dots ,k_1), \\
&\begin{matrix}
F_j(0,s)&=&P_{i_j}^2 \\
F_j(1,s)&=&Q_{\Phi_2^{\beta}(i_j)}^2
\end{matrix}
&(&j=k_1+1,\dots,k_2),  \\
&&\dots \\
&\begin{matrix}
F_j(0,s)&=&P_{i_j}^m  \\
F_j(1,s)&=&P_{\Phi_m^{\beta}(i_j)}^m
\end{matrix}
&(&j=k_{m-1}+1,\dots,n), 
\end{align*}
and, for each $s\in[0,1]$ if we define
\begin{equation*}
\beta^s=\{\beta_1^s,\dots,\beta_m^s\},
\end{equation*}
where
\begin{equation*}
\beta_j^s(t)=F_j(s,t) {\text{ for }}j=1,\dots,n,
\end{equation*}
then $\beta^s$ is itself a $\mathbb{P}M$-braid. 
\end{itemize}
\begin{example}
The following $\PM$-braids are equivalent. \\
\raisebox{40pt}{$\beta=$}
\includegraphics[width=4cm,bb=0 0 320 360]{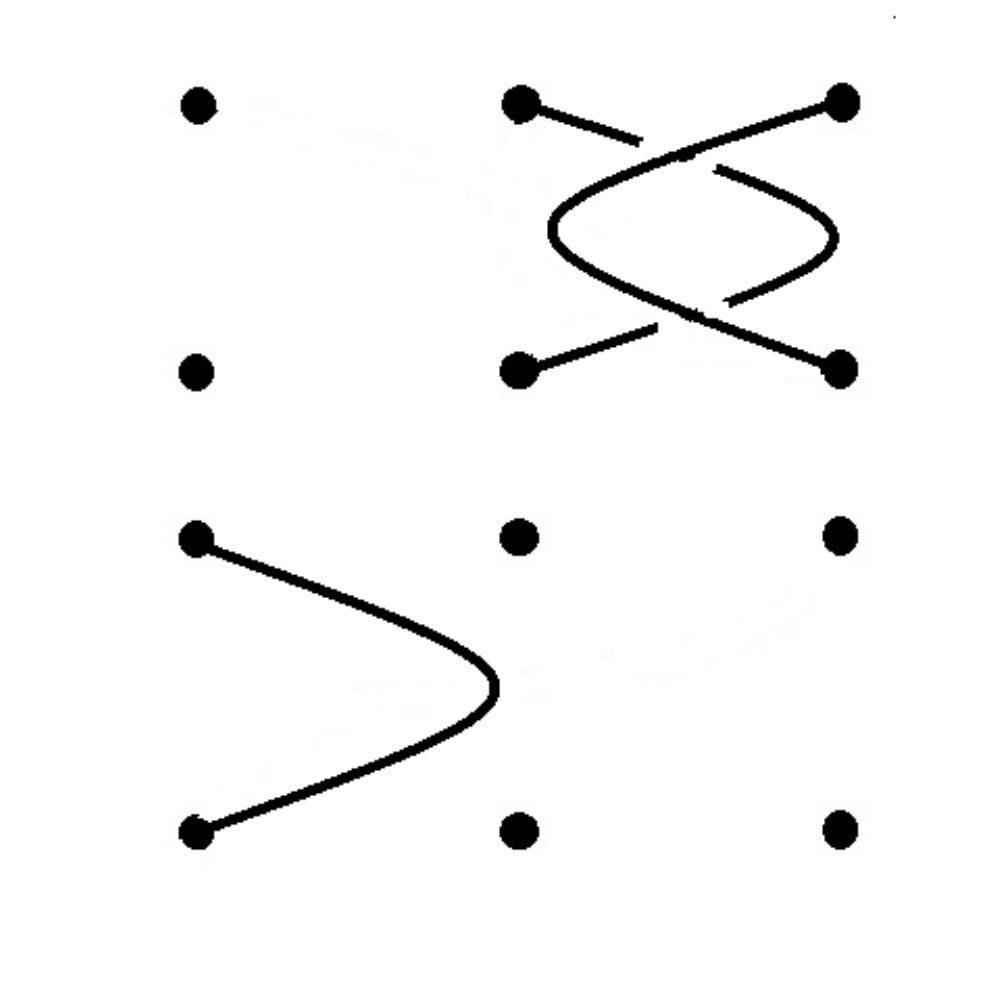} 
\raisebox{40pt}{$\gamma=$}
\includegraphics[width=4cm,bb=0 0 320 360]{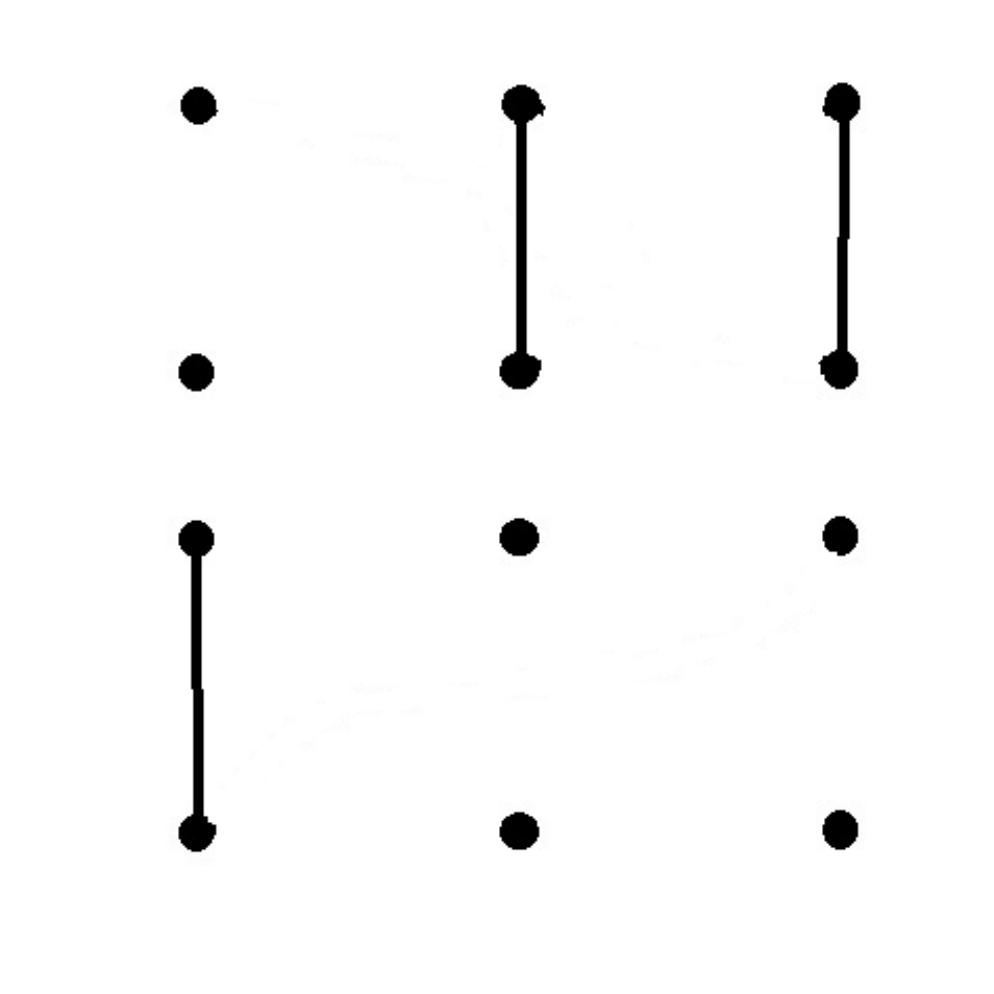} \\
\end{example}

Define the product $\beta\gamma$ of two braids
\begin{equation*}
\begin{split}
 &\beta=\{\beta_1,\dots,\beta_{k_1},\beta_{k_1+1},\dots,\beta_{k_2},\beta_{k_2+1},\dots,\beta_{k_{m-1}+1},\dots,\beta_n\}, \\
 &\gamma=\{\gamma_1,\dots,\gamma_{k_1},\gamma_{k_1+1},\dots,\gamma_{k_2},\gamma_{k_2+1},\dots,\gamma_{k_{m^{\prime}-1}+1},\dots,\gamma_n\}
 \end{split}
\end{equation*}
as follows. 

We first define an operation ($k_il_j$). Take $z_1^{(11)}>z_0^{(11)}>z_1^{(21)}>z_0^{(21)}>\dots>z_1^{(m1)}>z_0^{(m1)}>z_1^{(12)}>z_0^{(12)}>\dots>z_1^{(m2)}>z_0^{(m2)}>
\dots>z_1^{(mm^{\prime})}>z_0^{(mm^{\prime})}$.\\

($k_il_j$) : 
\begin{itemize}
\item[(1)] Translate $\{\gamma_{l_{j-1}+1},\dots,\gamma_{l_j}\}$ parallel to itself so that the upper plane of  $\{\gamma_{l_{j-1}+1},\dots,\gamma_{l_j}\}$ coincides with the lower plane of $\{\beta_{k_{i-1}+1},\dots,\beta_{k_j}\}$; 
\item[(2)] Translate the above system of arcs so that the upper plane of  $\{\beta_{k_{i-1}+1},\dots,\beta_{k_j}\}$ coincides with $z=z_1^{(ij)}$. Keeping   $z=z_1^{(ij)}$ fixed, contract the resulting systems of arcs so that the translated lower plane of $\{\gamma_{l_{j-1}+1},\dots,\gamma_{l_j}\}$ lies into the position of $z=z_0^{(ij)}$; 
\item[(3)] Remove any arc that do not now join the upper plane to the lower plane. 
\end{itemize}

Then take the operations ($k_1l_1$),$\dots$,($k_ml_1$),($k_1l_2$),$\dots$,($k_ml_2$),($k_1l_{m^{\prime}}$),$\dots$,($k_ml_{m^{\prime}}$), finally remove empty system of arcs. The resulting $\mathbb{P}M$-braid is denoted by $\beta\gamma$.

\begin{example}
Let $\beta$ and $\gamma$ be the following $\mathbb{P}M$-braids: \\
\raisebox{40pt}{$\beta=\begin{matrix}\beta_1 \\ \beta_2 \end{matrix}=$}
\includegraphics[width=4cm,bb=0 0 320 250]{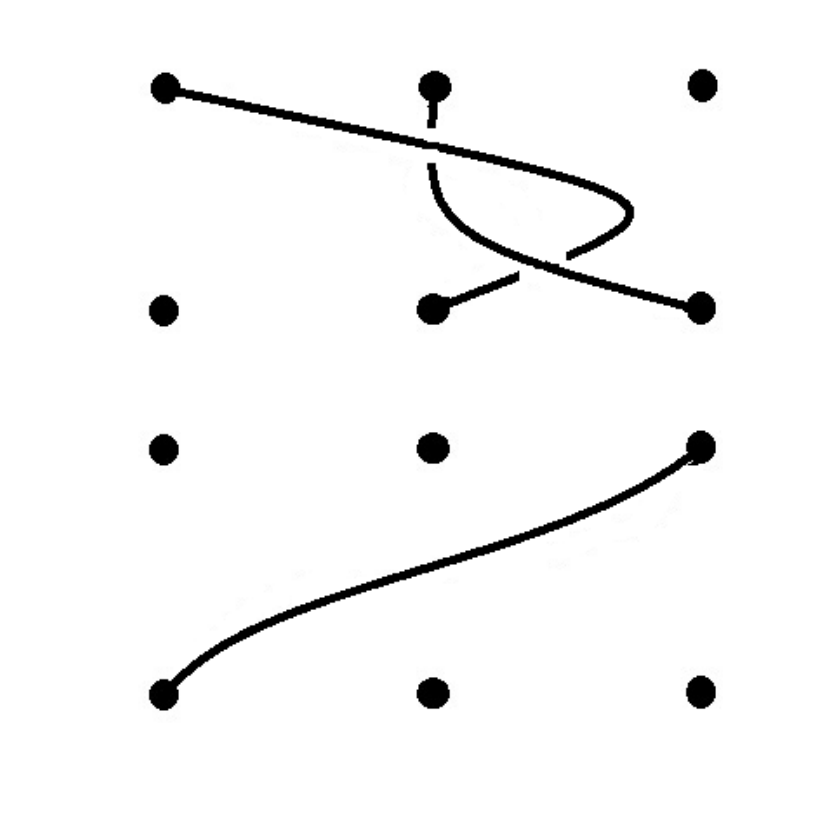} 
\raisebox{40pt}{$\gamma=\begin{matrix}\gamma_1 \\ \gamma_2 \end{matrix}=$}
\includegraphics[width=4cm,bb=0 0 320 250]{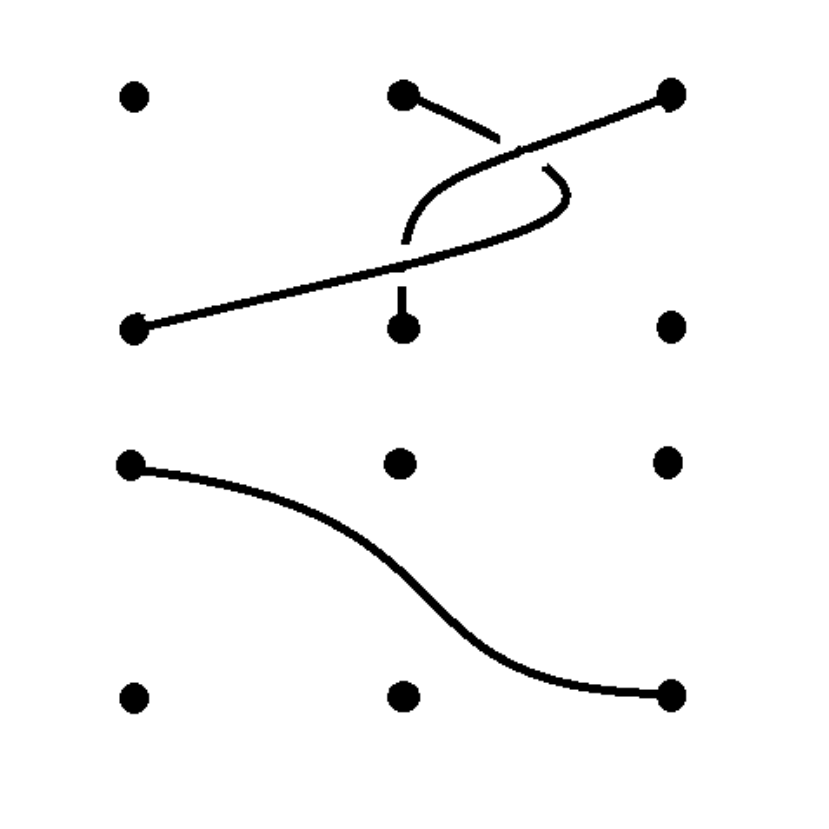} \\
then $\beta^2,\beta\gamma,\gamma\beta$ are obtained as follows: \\
\raisebox{140pt}{$\beta^2=\begin{matrix} \beta_1\beta_1 \\ \beta_2\beta_1 \\ \beta_1\beta_2 \\ \beta_2\beta_2 \end{matrix}=$}
\includegraphics[width=4cm,bb=0 0 320 450]{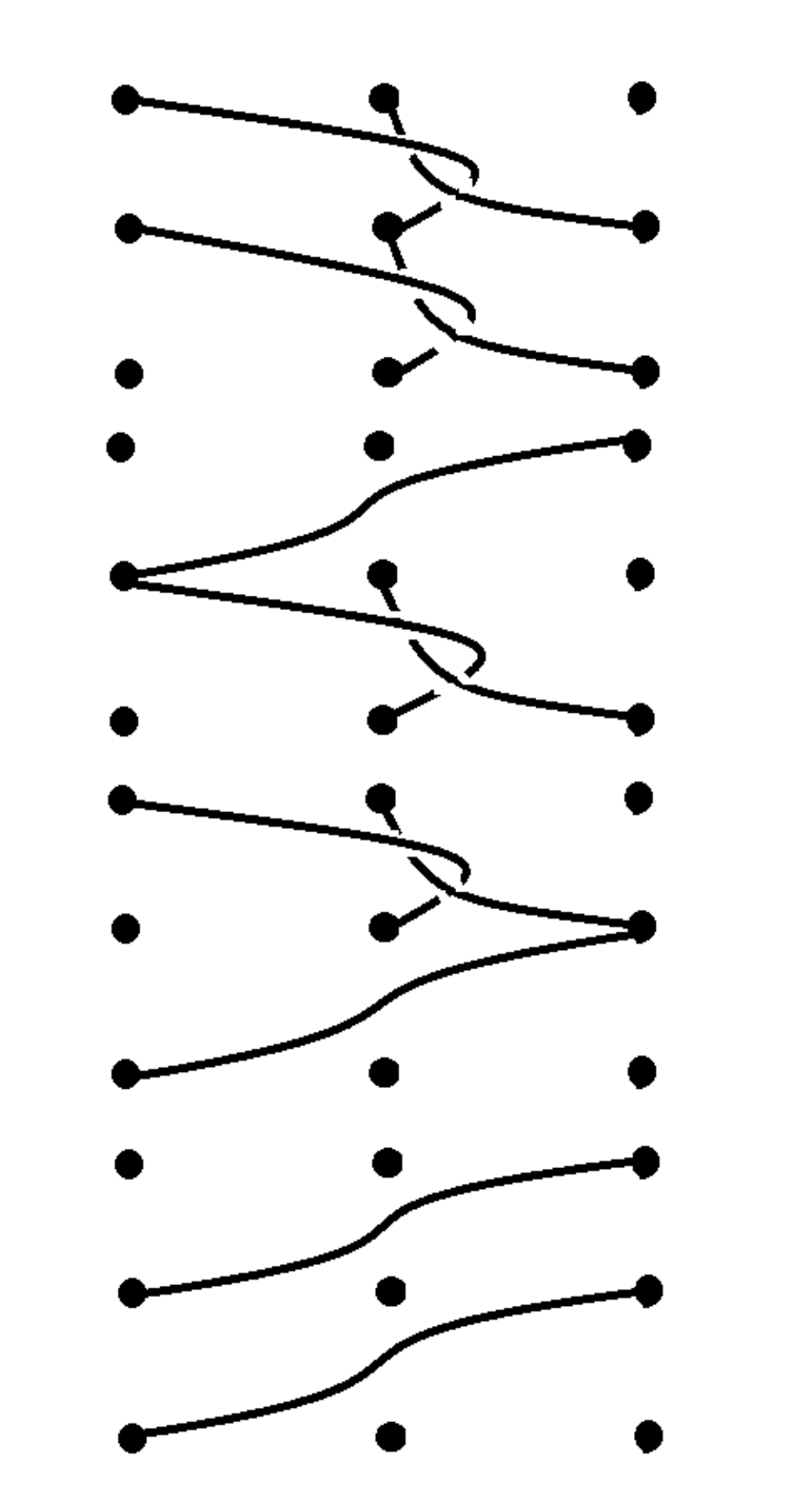}
\raisebox{140pt}{$=$}
\raisebox{70pt}{\includegraphics[width=4cm,bb=0 0 320 450]{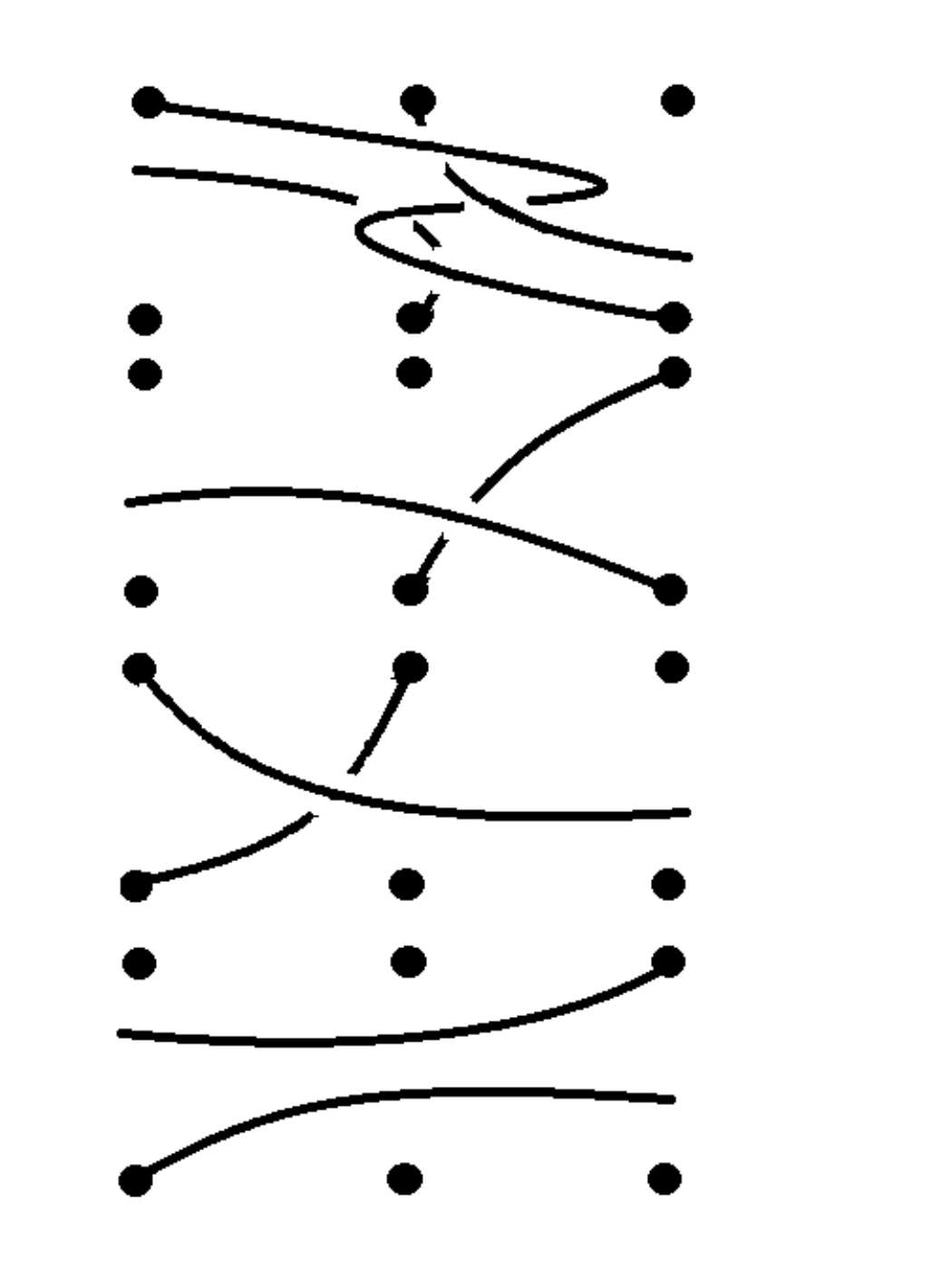}}
\raisebox{140pt}{$=$}
\raisebox{70pt}{\includegraphics[width=4cm,bb=0 0 320 350]{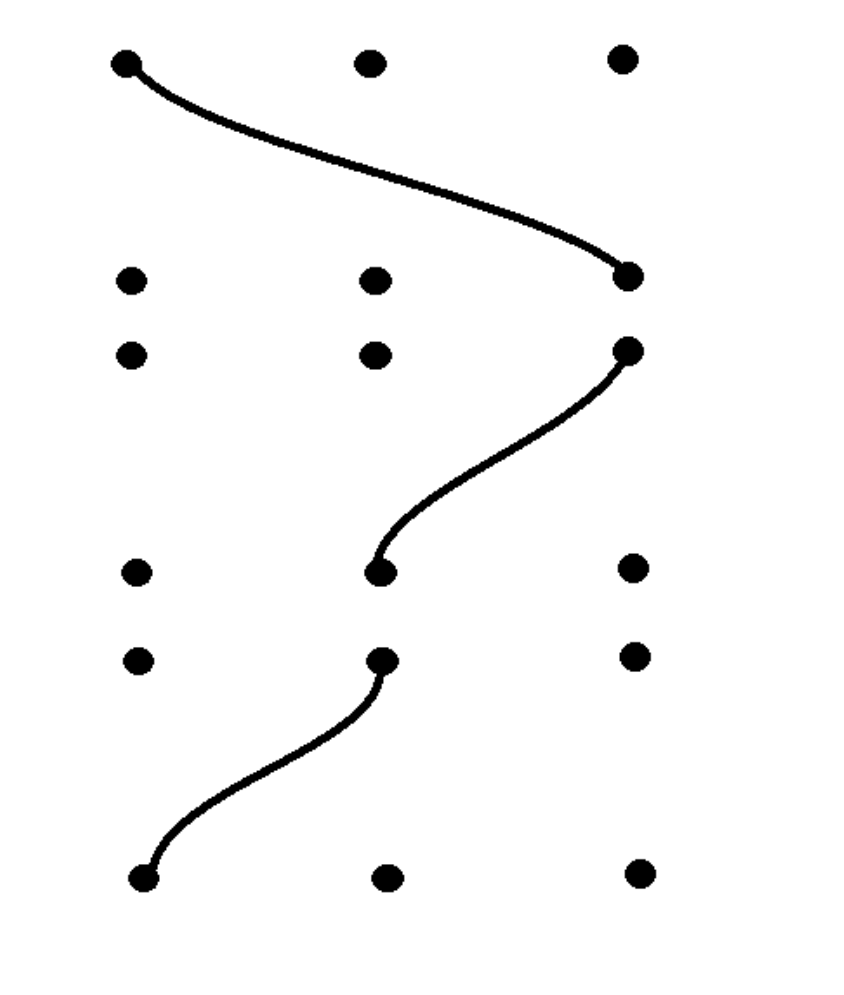}} \\
\raisebox{80pt}{$\beta\gamma=$}
\raisebox{50pt}{\includegraphics[width=4cm,bb=0 0 320 200]{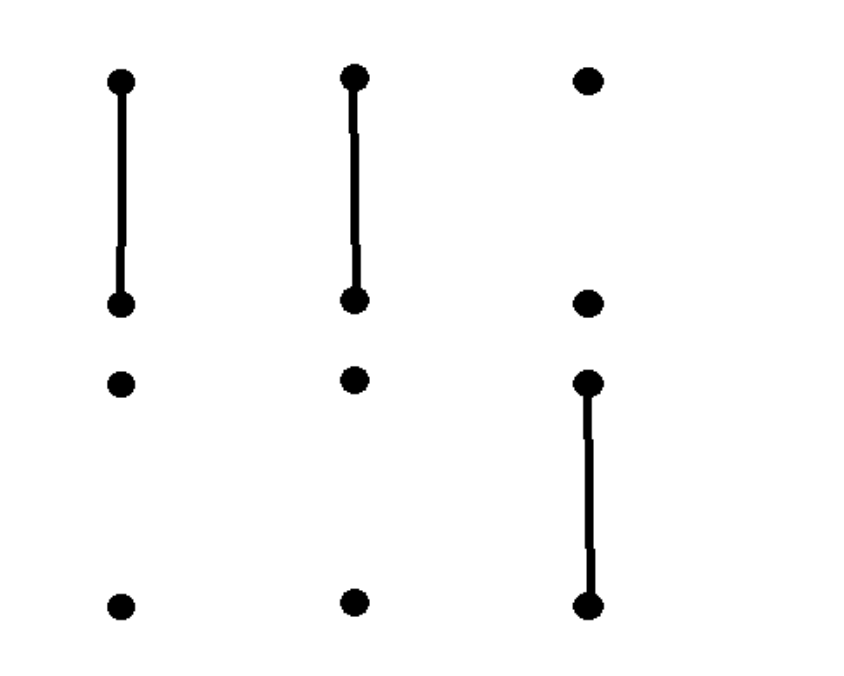}} 
\raisebox{80pt}{$\gamma\beta=$}
\raisebox{50pt}{\includegraphics[width=4cm,bb=0 0 320 200]{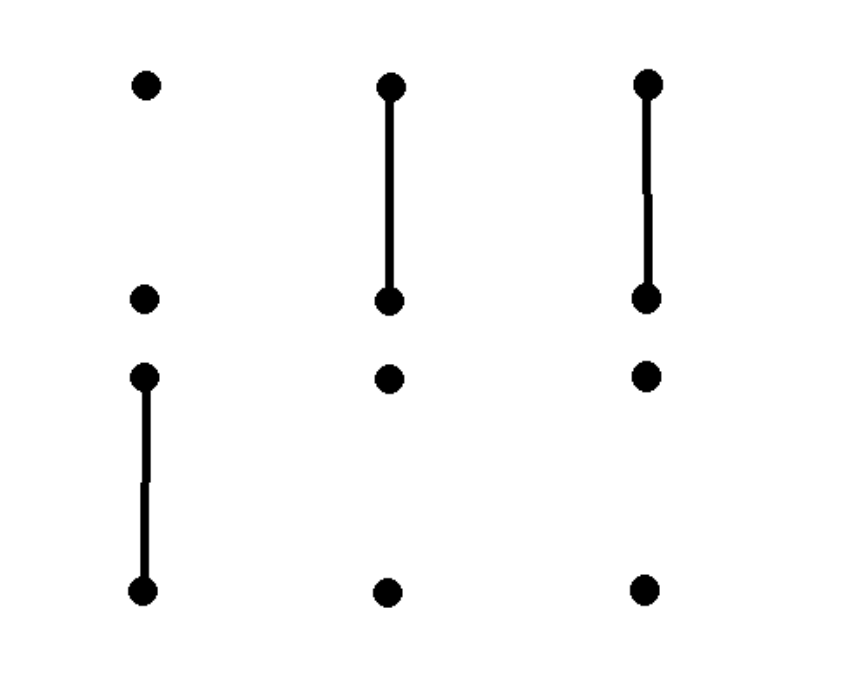}} 
\end{example}
We denote by $[\beta]$ the homotopy equivalence class of $\beta$. Put
\begin{equation*}
\mathscr{RB}_n=\{[\beta]:\beta {\text{ is a }} \mathbb{P}M{\text{-braid}}\}.
\end{equation*}
\begin{theorem} \label{pmb}
The braid $\mathbb{P}M$-monoid $\mathscr{M}$ is isomorphic to the monoid $\mathscr{RB}_n$.
\end{theorem}
\begin{proof}
Let $\Psi$ denote a map from the set of generators for $\mathscr{M}$ into $\mathscr{RB}_n$ given by \\
\raisebox{50pt}{$s_i\mapsto$}
\includegraphics[width=8cm,bb=0 0 320 150]{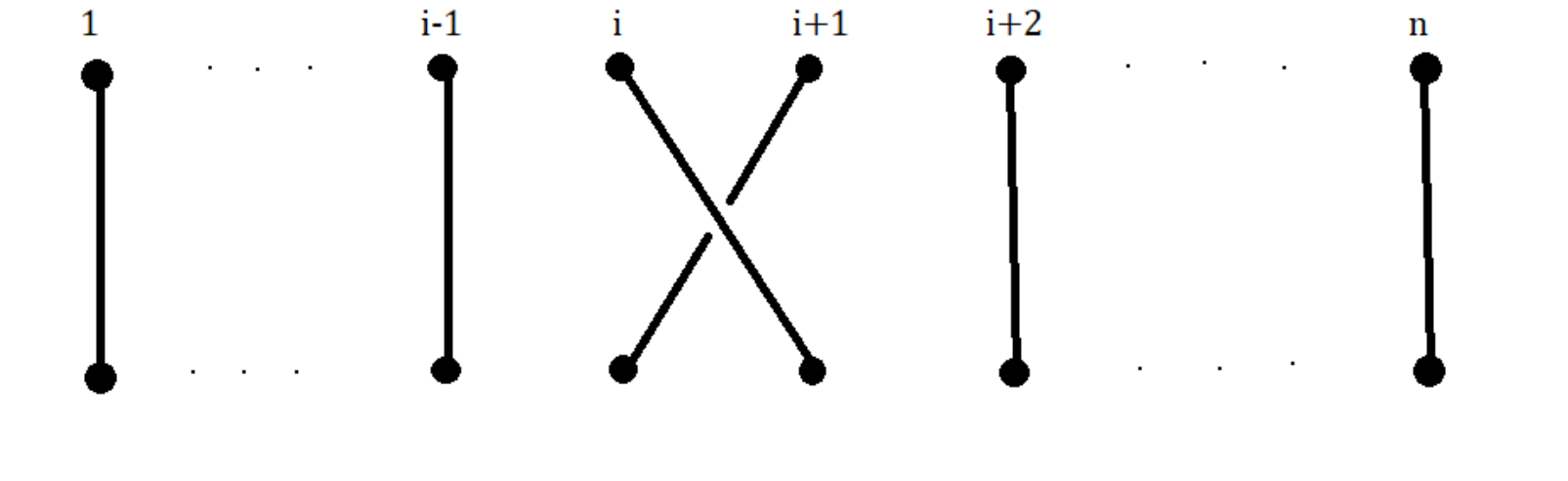}  \\
\raisebox{50pt}{$s_i^{-1}\mapsto$}
\includegraphics[width=8cm,bb=0 0 320 150]{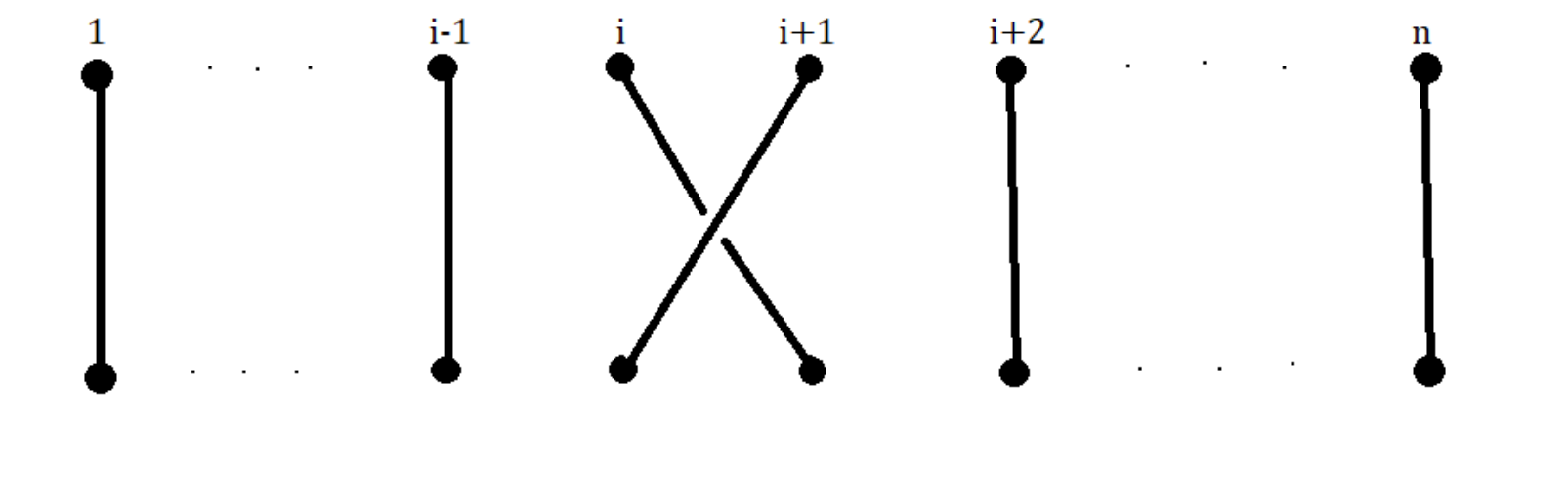} \\
\raisebox{180pt}{$e_{k_1,\dots,k_{m-1}}\mapsto$}
\includegraphics[width=6cm,bb=0 0 320 550]{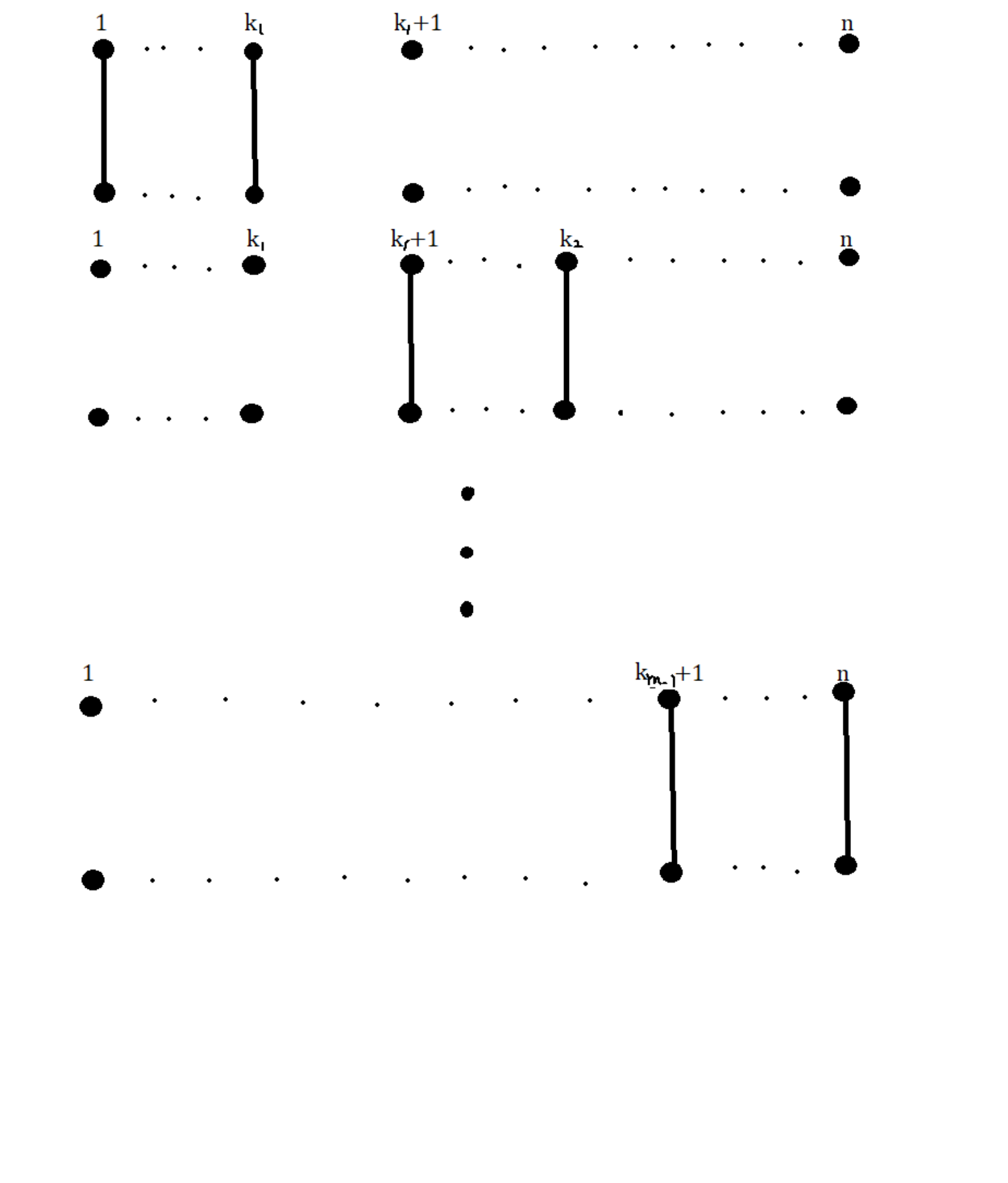} \\
The relations in the presentation for $\mathscr{M}$ hold for the images of the generators, so that $\Psi$ induces a well-defined homomorphism, which we also denote by $\Psi$. To prove the assertion, it suffices to prove that $\Psi$ is bijective. First we prove that $\Psi$ is surjective. Take any $\theta\in\mathscr{RB}_n$, and say $\theta$ has $k_1,k_2-k_1,\dots,n-k_m-1$-strings. Then there exist $s_{j_1},\dots,s_{j_t},s_{i_1},\dots,s_{i_r}$ and $e_{k_1,\dots,k_{m-1}}$ such that
\begin{equation*}
\theta=(\Psi(s_{j_1})\dots\Psi(s_{j_t}))^{-1}\Psi(e_{k_1,\dots,k_{m-1}})(\Psi(s_{j_1})\dots\Psi(s_{j_t}))(\Psi(s_{i_1})\dots\Psi(s_{i_1})).
\end{equation*}
Thus $\Psi$ is surjective. We next show that $\Psi$ is injective. By the discussion in the proof of Proposition \ref{rep},
\begin{equation*}
\mathscr{M}=\displaystyle\bigsqcup_{1\le k_1<\dots<k_{m-1}<n}B_ne_{k_1,\dots,k_{m-1}}B_n,
\end{equation*}
where $B_n$ is the Braid group. Let
\begin{equation*}
\Psi(b_1e_{k_1,\dots,k_{m-1}}b_2)=\Psi(b_1^{'}e_{l_1,\dots,l_{m^{'}-1}}b_2^{'}),
\end{equation*}
where $b_1,b_2,b_1^{'},b_2{'}\in B_n$. Since $\Psi$ is a homomorphism of monoids and $b_1,b_2,b_1^{'},b_2{'}$ have an inverse element, respectively, we can assume
\begin{equation*}
\Psi(b_1e_{k_1,\dots,k_{m-1}}b_2)=\Psi(e_{l_1,\dots,l_{m^{'}}-1}).
\end{equation*}
First, it must be $m=m^{'}$ and $k_1=l_1,\dots,k_{m-1}=l_{m-1}$. Since $\Psi(b_1)\Psi(b_2)|_{\{k_{l-1}+1,\dots,k_l\}}=id|_{\{k_{l-1}+1,\dots,k_l\}}$ for all $l=1,\dots,n$. Then as elements of the braid group $b_1b_2|_{\{k_{l-1}+1,\dots,k_l\}}=id|_{\{k_{l-1}+1,\dots,k_l\}}$. Therefore by the relation (\ref{re4-}), $b_1e_{k_1,\dots,k_{m-1}}b_2=e_{k_1,\dots,k_{m-1}}$. Thus $\Psi$ is injective.
\end{proof}
\subsection{Automorphism of free groups and word problems}
We will give a presentation of  $\mathbb{P}M$-braids by automorphism of free groups and find a solution to the word problem in $\mathscr{RB}_n$. We recall the cases : the classical braid group and an inverse braid monoid. \\ 
Let $F_n=F(x_1,\dots,x_n)$ be the free group of rank $n$ generated by $\{x_1,\dots,x_n\}$. For $1\le k\le n-1$, let $\tau_k:F_n\rightarrow F_n$ be the automorphism defined by 
\begin{equation*}
\tau_k:
\begin{cases}
x_k\mapsto x_k^{-1}x_{k+1}x_k \\
x_{k+1}\mapsto x_k \\
x_l\mapsto x_l & {\text{ if }} l\neq k,k+1 .
\end{cases}
\end{equation*}
Then the mapping $s_k\mapsto \tau_k$ $(1\le k\le n-1)$ determines a representation $\rho:B_n\rightarrow {\rm{Aut}}(F_n)$ called Artin representation. The following theorem  was proved by E. Artin.
\begin{theorem}[\cite{A1},\cite{A2}]
{\rm{(1)}} The Artin representation $\rho:B_n\rightarrow {\rm{Aut}}(F_n)$ is faithful.\\
{\rm{(2)}} An automorphism $\alpha\in {\rm{Aut}}(F_n)$ belongs to $\im\rho$ if and only if $\alpha(x_n\dots x_2x_1)=x_n\dots x_2x_1$ and there exists a permutation $\sigma\in S_n$ such that $\alpha(x_k)$ is conjugate to $x_{\sigma(k)}$ for all $1\le k\le n$.
\end{theorem}
The braid group $B_n$ can be viewed as a subgroup of ${\rm{Aut}}(F_n)$. Moreover this yields a solution to the word problem in $B_n$. 

Next recall the case of the inverse braid monoid studied by V. V. Vershinin \cite{V1}. Let $EF_n$ be a monoid of partial isomorphisms of a free group defined as follows. Let $a$ be an element of rook monoid $R_n$, and $J_k$ the image of  $a$. Let elements $i_1,\dots,i_k$ belong to the domain of definition of $a$. The monoid $EF_n$ consists of isomorphisms 
\begin{equation*}
F(x_{i_1},\dots,x_{i_k})\rightarrow F(x_{j_1},\dots,x_{j_k})
\end{equation*} 
expressed by 
\begin{equation*}
f_a(x_i)=
\begin{cases}
w_i^{-1}x_{a(i)}w_i & (i\in \{i_1,\dots,i_k\}) \\
{\rm{not\,defined}} & ({\rm{otherwise}}) \,\,\,\,\,\,\,\,\,\,\,\,\,\,\,.
\end{cases}
\end{equation*}
We define a map $\phi_n$ from $IB_n$ to $EF_n$ expanding the Artin representation $\rho$ by the condition that $\phi_n(e_j)$ as a partial isomorphism of $F_n$ is given by the formula  
\begin{equation*}
\phi_n(e_j)(x_i)=
\begin{cases}
x_i & (i\le j) \\
{\rm{not\,defined }} & (i>j).
\end{cases}
\end{equation*}
\begin{theorem}[\cite{V1} Theorem 2.2.]
The homomorphism $\phi_n$ is a monomorphism.
\end{theorem}
\begin{theorem}[\cite{V1} Theorem 2.3.]
The monomorphism $\phi_n$ gives a solution to the word problem for the inverse braid monoid.
\end{theorem}
Let $\mathscr{E}F_n$ be a monoid of sequence of partial isomorphisms of free group $F_n$ defined as follows. 

Let $\mathbb{A}=(\sigma,(\{i_1,\dots,i_{k_1}\},\dots,\{i_{k_{m-1}+1},\dots,n\}))\in \mathscr{R}_n$. The monoid $\mathscr{E}F_n$ consists of sequence of isomorphisms $\bm{f}_{\mathbb{A}}=(f_{A_1},\dots,f_{A_m})$, where for $j=1,\dots,m$
\begin{equation*}
f_{A_j}:F(x_{i_{k_{j-1}+1}},\dots,x_{i_{k_j}})\rightarrow F(x_{\sigma(i_{k_{j-1}+1})},\dots,x_{\sigma(i_{k_j})})
\end{equation*}
is defined by
\begin{equation*}
f_{A_j}(x_l)=
\begin{cases}
w_l^{-1}x_{A_j(l)}w_l & (l\in \{i_{k_{j-1}+1},\dots,i_{k_j}\}) \\
{\rm{not \, defined}} & ({\rm{otherwise}})\,\,\,\,\,\,\,\,\,\,\,\,\,\,\,\,\,\,\,\,\,\,\,\,\,\,\,, 
\end{cases}
\end{equation*}
where $w_i$ is a word on $x_{\sigma(i_{k_{j-1}+1})},\dots,x_{\sigma(i_{k_j})}$. We define a map $\varphi_n$ from $\mathscr{RB}_n$ to $\mathscr{E}F_n$ extending the Artin representation $\rho$ by the condition that $\varphi_n(e_{k_1,\dots,k_{m-1}})$ as a sequence of partial isomorphisms of $F_n$ is given by the formula 
\begin{equation*}
\varphi_n(e_{k_1,\dots,k_{m-1}})(x_i)=(f_1(x_i),\dots,f_{m-1}(x_i)),
\end{equation*}
where 
\begin{equation*}
f_j(x_i)=
\begin{cases}
x_i & (k_{j-1}\le i\le k_j) \\
{\text{not  defined}} & (i<k_{j+1},i>k_j)
\end{cases}
\end{equation*}
for $j=1,\dots,m-1$. 
\begin{proposition}
The homomorphism $\varphi_n$ is a monomorphism.
\end{proposition}
\begin{proof}
Let $\mathscr{RB}_n^{(m)}$ be the set of $\mathbb{P}M$-braids which have $m$ layers. Then as a set we have the following decomposition :
\begin{equation*}
\begin{split}
\mathscr{RB}_n&=\displaystyle \bigsqcup_{m\geq 1}\mathscr{RB}_n^{(m)} \\
&=\displaystyle \bigsqcup_{m\geq 1} \displaystyle \bigsqcup_{(p_1,\dots,p_m)\in P_n,\sigma\in S_n}B(p_1,\sigma(p_1))\times\dots\times B(p_m,\sigma(p_m)),
\end{split}
\end{equation*}
where $B(p_i,\sigma(p_i))$ is the braid group starting at $p_i$ and ending at $\sigma(p_i)$. Let $\#p_i=r_i$. Then consider the following diagram:
\[
\xymatrix{
B_{r_1}\times\dots\times B_{r_m}\ar[r]^{Id\times\dots\times Id}\ar[d]_{\rho_1\times\dots\times\rho_m}&B_{r_1}\times\dots\times B_{r_m}\ar[d]^{\psi(p_1,\sigma(p_1))\times\dots\times\psi(p_m,\sigma(p_m))} \\
B(p_1,\sigma(p_1))\times\dots\times B(p_m,\sigma(p_m))\ar[r]^-{\varphi_n}&\mathscr{E}F_n .
}
\]
The above diagram is commutative since the diagram $(2.6)$ in \cite{V1} is commutative. Thus $\varphi_n$ is a monomorphism.
\end{proof}
Suppose that we are given a group or a monoid presented by generators and relations $\langle X\mid R\rangle$. Let $\beta_1,\beta_2\in\langle X\mid R\rangle$. A solution to the word problem of $\langle X\mid R\rangle$ is a strategy if we can judge $\beta_1=\beta_2$ or not in $\langle X\mid R\rangle$. 
\begin{theorem} \label{wp}
The morphism $\phi_n$ gives a solution to the word problem for the braid $\mathbb{P}M$-monoid.
\end{theorem}
\begin{proof}
This assertion holds by the following fact : Two words represent the same element of the monoid if and only if they have the same action on the finite set of generators of the free group.
\end{proof}

\thanks{Acknowledgements.}
I would like to thank Professor Mutsumi Saito and Professor Youichi Shibukawa for informing the matched pairs. I also thank Masamitsu Aoki, Kazuya Takasaki and Toshifumi Yabu for advising helpful comments. 

\end{document}